\documentclass[12pt]{amsart}
\usepackage[widespace]{fourier}
\usepackage{amssymb}
\usepackage{amsthm}
\usepackage{graphicx}
\usepackage[all,2cell]{xy}
\usepackage{verbatim}
\usepackage{tikz}
\usepackage{tikz-cd}
\usepackage{mathtools}
\usetikzlibrary{matrix,arrows,decorations.pathmorphing}
\numberwithin{equation}{section}
\newtheorem{theorem}{Theorem}[section]
\newtheorem{corollary}[theorem]{Corollary}

\newtheorem{example}[theorem]{Example}
\newtheorem{lemma}[theorem]{Lemma}

\newtheorem{proposition}[theorem]{Proposition}

\newtheorem*{theorem*}{Theorem}
\usepackage{lipsum}

\title[Rigidity of BSDH variety for  $PSp(2n, \mathbb C)$]{Rigidity of Bott-Samelson-Demazure-Hansen variety for  $PSp(2n, \mathbb C)$}                                     %<-------------------
\author[B.N.Chary and  S.S.Kannan]{B. Narasimha Chary and S. Senthamarai Kannan}                
\address{%
B. Narasimha Chary\\
Chennai Mathematical Institute, \\
Plot H1, SIPCOT IT Park, Siruseri, \\
Kelambakkam, 603103, India.\\
chary@cmi.ac.in.}

\address{%
S. Senthamarai Kannan\\
Chennai Mathematical Institute, \\
Plot H1, SIPCOT IT Park, Siruseri, \\
Kelambakkam, 603103, India.\\
kannan@cmi.ac.in.}


\begin{document}

\maketitle

\begin{abstract}
Let $G=PSp(2n, \mathbb C)(n\geq 3)$ and $B$ be a Borel subgroup of $G$ containing a maximal torus $T$ of $G$. Let $w$ be an element 
of the Weyl group $W$ and $X(w)$ be the Schubert variety in the flag variety $G/B$
corresponding to $w$. Let $Z(w,\underline i)$ be the Bott-Samelson-Demazure-Hansen 
variety (the desingularization of $X(w)$) corresponding to a reduced expression $\underline i$ of $w$. 

In this article, we study the cohomology groups of the tangent bundle on $Z(w_0, \underline i)$, where $w_0$ is the 
longest element of the Weyl group $W$. We describe  all the
 reduced expressions  $\underline i$ of $w_0$ in terms of a Coxeter element 
 such that all the higher cohomology groups  of the tangent bundle on $Z(w_0, \underline i)$ vanish (see Theorem \ref{Theorem}).  
\end{abstract}

\section{Introduction}
Let $G$ be a simple algebraic group of adjoint type over  the field $\mathbb{C}$ of 
complex numbers. We fix a maximal torus $T$ of $G$ and 
let $W = N_G(T)/T$ denote the Weyl group of $G$ with respect to $T$.
We denote by $R$ the set of roots of $G$ with respect to $T$ and by
 $R^{+}$ a set of positive roots. Let $B^{+}$ be the Borel 
subgroup of $G$ containing $T$ with respect to $R^{+}$. 
Let 
$w_0$ denotes the longest element of the Weyl group $W$. Let $B$ be the 
Borel subgroup of $G$ opposite to $B^+$ determined by $T$, i.e. $B=n_{w_0}B^+n_{w_0}^{-1}$, where $n_{w_0}$ is a representative of $w_0$ in $N_G(T)$.
Note that the 
roots of 
$B$ is the set $R^{-} :=-R^{+}$ of negative roots. We use the notation $\beta<0$ for 
$\beta \in R^{-}$.
Let $S = \{\alpha_1,\ldots,\alpha_n\}$ 
denote the set of all simple roots in $R^{+}$, where $n$ is the rank of $G$. 
The simple reflection in the Weyl group corresponding to a simple root $\alpha$ is denoted
by $s_{\alpha}$. For simplicity of notation, the simple reflection corresponding to a simple root $\alpha_{i}$ is denoted
by $s_{i}$. 

For $w \in W$, let $X(w)(:=\overline{BwB/B})$ denote the Schubert variety in the flag variety
$G/B$ corresponding to $w$. Note that in general Schubert varieties are not smooth. 

Given a reduced expression 
$w=s_{i_1}s_{i_2}\cdots s_{i_r}$ of $w$, with 
the corresponding tuple $\underline i:=(i_1,\ldots,i_r)$, we denote by 
$Z(w,{\underline i})$ the desingularization  of the Schubert variety $X(w)$, 
which is now known as Bott-Samelson-Demazure-Hansen (for short BSDH) variety. 

Demazure in  \cite{Dem1} and 
Hansen in \cite{Hansen} independently constructed these desingularizations of Schubert varieties 
using the idea from \cite{Bott-Sam}.
In this paper we prove that two BSDH-varieties $Z(w,\underline i)$ and $Z(w,\underline j)$ are isomorphic if $\underline i$ and $\underline j$
 differ only by commuting relations (see Theorem \ref{com}).

In \cite{Bott1}, it is proved that all the higher cohomology groups 
$H^{i}(G/B, T_{G/B})$ for the tangent bundle   
$T_{G/B}$ on  $G/B$ vanish. 
In \cite{Ka4}, it is proved  that the higher
cohomology groups  of the
 restriction of  $T_{G/B}$ to $X(w)$ vanish (see \cite[Theorem 4.1 and  Theorem 6.5]{Ka4}).
  
In \cite{CKP}, we  
proved the following vanishing results of the tangent bundle $T_{Z(w, \underline i)}$ on $Z(w, \underline i)$
(see \cite[Section 3]{CKP});\\
(1) $H^j(Z(w, \underline i), T_{Z(w, \underline i)})=0$ for all $j\geq 2$.\\
(2) If $G$ is simply laced, then $H^j(Z(w, \underline i), T_{Z(w, \underline i)})=0$ for all
  $j\geq 1$.

As a consequence, it follows that the BSDH-varieties are rigid 
for simply laced groups and their deformations are  unobstructed in general 
(see \cite[Section 3]{CKP}). 
 The above vanishing result is independent of the choice of the reduced expression $\underline i$ of $w$.
 While computing the first cohomology group $H^1(Z(w, \underline i),  T_{Z(w, \underline i)})$ for non simply laced group,
 we observed that this cohomology group very much depend on the choice of a reduced expression $\underline i$ of $w$.

It is a natural question to ask  that for which reduced expressions $\underline i$ of $w$, 
the cohomology group $H^1(Z(w, \underline i), T_{Z(w, \underline i)})$  does vanish ?
In this article, we give a partial answer to this question for $w=w_0$ when $G=PSp(2n, \mathbb C)$. 

Recall that a Coxeter element is an element of the Weyl group 
having a reduced expression of the form $s_{i_{1}}s_{i_{2}} \cdots s_{i_{n}}$
such that $i_{j}\neq i_{l}$ whenever $j\neq l$ (see \cite[p. 56, Section 4.4]{Hum0}).
Note that for any Coxeter element $c$, there is a decreasing sequence of integers $n\geq a_1> a_2 > \ldots > a_k=1$
 such that $c=\prod\limits_{j=1}^{k}[a_j, a_{j-1}-1]$, where $a_0:=n+1$, $[i, j]:=s_{i}s_{i+1}\cdots s_{j}$ for $i\leq j$.

{\it  Throughout this paper, for simplicity, we denote the product $w_{1}\cdots w_{k}$ of elements in 
$W$ by $\prod\limits_{j=1}^{k}w_{j}$.   For instance, we use the product $c=\prod\limits_{j=1}^{k}[a_j, a_{j-1}-1]$ as above to denote  $c=[a_1, a_{0}-1] [a_{2}, a_{1}-1]\cdots [a_{k} , a_{k-1}-1]$. }

In this paper we prove the following theorem. 
 \begin{theorem*}
 
Let $G=PSp(2n, \mathbb C)$ $(n\geq 3)$ and let $c\in W$ be a Coxeter element.
Let $\underline i=(\underline i^1, \underline i^2,\ldots, \underline i^n)$ be a sequence corresponding 
to a reduced expression of $w_0$, where
$\underline i^r$ $(1\leq r\leq n)$ is a sequence of reduced expressions of $c$ (see Lemma \ref{coxeter}). 
  Then,  
  $H^j(Z(w_0, \underline i), T_{(w_0, \underline i)})=0$ for all $j\geq 1$ if and only if $c=\prod\limits_{j=1}^{k}[a_j, a_{j-1}-1]$, where $a_0:=n+1$
and $a_j\neq n-1$ for every $j=1,2,\ldots, k$.
\end{theorem*}

By the above vanishing results, we conclude that if $G=PSp(2n, \mathbb C)$ $(n\geq 3)$ and $\underline i=(\underline i^1, \underline i^2,\ldots, \underline i^n)$
is a reduced expression of $w_0$ as above, then the BSDH-variety $Z(w_0, \underline i)$ is rigid.

The organization of the paper is as follows:
In Section 2, we recall some preliminaries on BSDH-varieties. In Section 3, we prove that two
BSDH-varieties $Z(w,\underline i)$ and $Z(w,\underline j)$ are isomorphic if $\underline i$ and $\underline j$
 differ only by commuting relations.
 We deal with the special case $G=PSp(2n, \mathbb C)$  $(n\geq 3)$ in sections 4, 5, 6, 7 and 8.
In Section 4, we write down explicit reduced expressions of $c^i (1\leq i \leq n)$ for each Coxeter element $c$. 
In Section 5 (respectively, Section 6) we compute the weight spaces of $H^0$ (respectively, $H^1$) of the relative tangent bundle of BSDH-varieties
associated to some elements of the Weyl group. 
In Section 7, we prove some results on cohomology modules of tangent bundle of BSDH varieties.
In Section 8, we prove the  main result using the results from the previous sections.

{\bf Acknowledgements.} We are very grateful to the referee for his/her comments which helps to improve the exposition of the paper..
We would like to thank A.J. Parameswaran for the useful discussions. We would like to thank the 
Infosys Foundation for the partial financial support.

\section{Preliminaries}
We refer to \cite{Bri}, \cite{Hart}, \cite{Hum1}, \cite{Hum2}, \cite{Jan} and \cite{Springer} for preliminaries in Algebraic geometry, Algebraic groups and Lie algebras.
For a simple root $\alpha\in S$, we denote by
$P_{\alpha}$ the minimal parabolic subgroup of $G$ containing $B$ and
$s_{\alpha}$.  
We recall that the BSDH-variety corresponds to a reduced expression $\underline i$ of
$w=s_{i_1}s_{i_2}\cdots s_{i_r}$  is defined by
$$
Z(w, \underline i) = \frac {P_{\alpha_{i_{1}}} \times P_{\alpha_{i_{2}}} \times \cdots \times 
P_{\alpha_{i_{r}}}}{B \times \cdots
\times B},
$$

where the action of $B \times \cdots \times B$ on $P_{\alpha_{i_{1}}} \times P_{\alpha_{i_{2}}}
\times \cdots \times P_{\alpha_{i_{r}}}$ is given by $(p_1, \ldots , p_r)(b_1, \ldots
, b_r) = (p_1 \cdot b_1, b_1^{-1} \cdot p_2 \cdot b_2, \ldots
,b^{-1}_{r-1} \cdot p_r \cdot b_r)$, $ p_j \in P_{\alpha_{i_{j}}}$, $b_j \in B$ and 
$\underline i=(i_1, i_2, \ldots, i_r)$
(see \cite[p.73, Definition 1]{Dem1}, \cite[p.64, Definition 2.2.1]{Bri}).

We note that for each reduced expression $\underline i$ of $w$, $Z(w, \underline i)$ is a 
smooth projective 
variety. We denote by $\phi_w$, the natural birational surjective morphism
from $ Z(w, \underline i)$ to $X(w)$.

Let $f_r : Z(w, \underline i) \longrightarrow Z(ws_{i_r},
\underline i')$ denote the map induced by the
projection $P_{\alpha_{i_1}} \times P_{\alpha_{i_2}} \times \cdots \times 
P_{\alpha_{i_r}} \longrightarrow P_{\alpha_{i_1}} \times P_{\alpha_{i_2}}
\times \cdots \times P_{\alpha_{i_{r-1}}}$, where $\underline i'=(i_1,i_2,\ldots, i_{r-1})$.
Then we observe that $f_r$ is a 
$P_{\alpha_{i_r}}/B \simeq {\mathbb P}^{1}$-fibration.

For a $B$-module $V$, let ${\mathcal L}(w,V)$ denote the restriction of the
associated  homogeneous  vector bundle on $G/B$ to $X(w)$.
By abuse of notation, we denote the pull back of ${\mathcal L}(w,V)$ via $\phi_w$ to 
$Z(w, \underline i)$ also by ${\mathcal L}(w,V)$, when there is no cause 
for confusion. 
Since for any $B$-module $V$, the vector bundle ${\mathcal L}(w,V)$
on $Z(w, \underline i)$ is the pull back of the homogeneous vector 
bundle from $X(w)$, we conclude that the cohomology modules
$$H^{j}(Z(w, \underline i) ,~{\mathcal L}(w,V))\cong H^{j}(X(w),
~{\mathcal L}(w,V))$$ for all $j\geq 0$ (see  \cite[Theorem 3.3.4 (b)]{Bri}), are independent
 of the choice of the reduced expression $\underline i$. Hence we denote
$H^{j}(Z(w, \underline i) ,~{\mathcal L}(w,V))$ by $H^j(w,V)$. In particular,
if $\lambda$ is a character of $B$, then we denote the cohomology modules 
$H^{j}(Z(w, \underline i) ,~{\mathcal L}_{\lambda} )$ by $H^j(w, \lambda)$. 

 We recall the following short exact sequences of $B$-modules from \cite{CKP}, we call it
{\it SES}. 
\begin{enumerate}
 \item  $H^0(w, V)\simeq H^0(s_{\gamma}, H^0(s_{\gamma}w, V))$.
 \item  $0 \to H^{1}(s_{\gamma}, H^{0}(s_{\gamma}w, V)) 
\to 
H^{1}(w, V) \to
H^{0}(s_{\gamma} , H^{1}(s_{\gamma}w, V) ) \to 0$.
\end{enumerate}
Let $\alpha$ be a simple root and $\lambda \in X(T)$ be such that 
$\langle \lambda , \alpha \rangle  \geq 0$. Let $\mathbb{C}_{\lambda}$ 
denote the one dimensional $B$-module associated to $\lambda$. Here, we recall the following 
result
 due to Demazure \cite[Page 1]{Dem} on a short exact sequence of $B$-modules:
\begin{lemma}\label{dem1} Let $\alpha$ be a simple root and $\lambda \in X(T)$ be such that 
$\langle \lambda , \alpha \rangle  \geq 0$.
Let $ev:H^0({s_\alpha, \lambda}) 
\longrightarrow \mathbb{C}_\lambda$  be the evaluation map. 
 Then we have
\begin{enumerate}
\item If $\langle \lambda,\alpha \rangle=0$, then $H^0({s_\alpha, \lambda})\simeq\mathbb{C}_\lambda$.
\item If  $\langle \lambda,\alpha\rangle \geq 1$, then $\mathbb C_{s_{\alpha}(\lambda)}\hookrightarrow H^0({s_\alpha, \lambda})$
and there is a short exact sequence of $B$-modules:

 $0\longrightarrow  H^0({s_\alpha,\lambda-\alpha})\longrightarrow H^0({s_\alpha, \lambda})/ \mathbb C_{s_{\alpha}(\lambda)}\overset{ev}{\longrightarrow} \mathbb C_{\lambda}\longrightarrow 0.$
Further more, $H^0({s_\alpha,\lambda-\alpha})=0$ when $\langle\lambda,\alpha \rangle=1$.
\item Let $n=\langle \lambda, \alpha \rangle$. 
As a $B$-module, $H^0(s_{\alpha}, \lambda)$ has a composition series
$$0\subsetneq V_n\subsetneq V_{n-1}\subsetneq \ldots \subsetneq V_0=H^0(s_{\alpha}, \lambda)$$ 
such that  $V_i/V_{i+1}\simeq \mathbb C_{\lambda -i\alpha}$
for $i=0,1, \ldots, n-1$ and  $V_n=\mathbb C_{s_{\alpha}(\lambda)}$.
\end{enumerate}
\end{lemma}

 We define the {\rm dot} action by $w\cdot\lambda=w(\lambda+\rho)-\rho$, where $\rho$ is the half sum of positive roots.
 As a consequence of the exact sequences of Lemma \ref{dem1}, we can prove 
the following.
\begin{lemma} \label{lemma2.2}
Let $w = \tau s_{\alpha}$, $l(w) = l(\tau)+1$. Then we have 
\begin{enumerate}
\item  If $\langle \lambda , \alpha \rangle \geq 0$, then 
$H^{j}(w , \lambda) = H^{j}(\tau, H^0({s_\alpha, \lambda}) )$ for all $j\geq 0$.
\item  If $\langle \lambda ,\alpha \rangle \geq 0$, then $H^{j}(w , \lambda ) =
H^{j+1}(w , s_{\alpha}\cdot \lambda)$ for all $j\geq 0$.
\item If $\langle \lambda , \alpha \rangle  \leq -2$, then $H^{j+1}(w , \lambda ) =
H^{j}(w ,s_{\alpha}\cdot \lambda)$ for all $j\geq 0$. 
\item If $\langle \lambda , \alpha \rangle  = -1$, then $H^{j}( w ,\lambda)$ 
vanishes for every $j\geq 0$.
\end{enumerate}
\end{lemma}
The following consequence of  Lemma \ref{lemma2.2} will be used 
to compute cohomology modules in this paper. Now onwards we will denote the Levi subgroup of $P_{\alpha}$  ( $\alpha\in S$ ) containing $T$ by $L_{\alpha}$  and the 
subgroup $L_{\alpha}\cap B$ by $B_{\alpha}$.  Let $\pi:\widetilde{G}\longrightarrow G$ be the universal cover.  Let $\widetilde L_{\alpha}$    ( respectively, $\widetilde B_{\alpha}$ )  be the inverse image of  $L_{\alpha}$    (  respectively, of
$B_{\alpha}$ ). 

\begin{lemma}\label{lemma2.3}
Let $V$ be an irreducible  $L_{\alpha}$-module. Let $\lambda$
be a character of $B_{\alpha}$. Then, we have 
\begin{enumerate}
\item If
$\langle \lambda , \alpha \rangle \geq 0$, then 
$H^{0}(L_{\alpha}/B_{\alpha} , V\otimes \mathbb{C}_{\lambda})$ 
is isomorphic as an $L_{\alpha}$-module to the tensor product of $V$ and 
$H^{0}(L_{\alpha}/B_{\alpha} , \mathbb{C}_{\lambda})$, and 
$H^{j}(L_{\alpha}/B_{\alpha} , V\otimes \mathbb{C}_{\lambda}) =0$ 
for every $j\geq 1$.
\item If
$\langle \lambda , \alpha \rangle  \leq -2$, 
$H^{0}(L_{\alpha}/B_{\alpha} , V\otimes \mathbb{C}_{\lambda})=0$, 
and $H^{1}(L_{\alpha}/B_{\alpha} , V\otimes \mathbb{C}_{\lambda})$
is isomorphic to the tensor product of  $V$ and $H^{0}(L_{\alpha}/B_{\alpha} , 
\mathbb{C}_{s_{\alpha}\cdot\lambda})$. 
\item If $\langle \lambda , \alpha \rangle  = -1$, then 
$H^{j}( L_{\alpha}/B_{\alpha} , V\otimes \mathbb{C}_{\lambda})=0$ 
for every $j\geq 0$.
\end{enumerate}
\end{lemma}
Recall the structure of indecomposable 
$B_{\alpha}$-modules  ( respectively, $\widetilde {B}_{\alpha}$-modules ) ( see \cite[p.130, Corollary 9.1]{Ka1} ).
\begin{lemma}\label{lemma2.4}{\ }

\begin{enumerate}
 \item
Any finite dimensional indecomposable $\widetilde{B}_{\alpha}$-module $V$ is isomorphic to 
$V^{\prime}\otimes \mathbb{C}_{\lambda}$ for some irreducible representation
$V^{\prime}$ of $\widetilde{L}_{\alpha}$ and for some character $\lambda$ of $\widetilde{B}_{\alpha}$.
\item
Any finite dimensional indecomposable $B_{\alpha}$-module $V$ is isomorphic to 
$V^{\prime}\otimes \mathbb{C}_{\lambda}$ for some irreducible representation
$V^{\prime}$ of $\widetilde{L}_{\alpha}$ and for some character $\lambda$ of $\widetilde{B}_{\alpha}$.
\end{enumerate}
\end{lemma}

Recall the following result from \cite{Ka4}(see \cite[Corollary 5.6]{Ka4}).
\begin{corollary} \label{short}
  Let $\alpha$ be a short root such that $-\alpha \notin S$, let $w\in W$. Then we have, 
  $H^i(w, \alpha)=0$ for $i\geq 1$.
 \end{corollary}
 
 \section{Reduced expressions differing only by commuting relations}
  Let $w\in W$ and let $\underline i:=(i_1, \ldots, i_{r})$ be the tuple corresponding to a reduced expression 
  $w=s_{i_{1}}s_{i_{2}}\cdots s_{i_{r}}$ of $w$. Note that if $\langle \alpha_{i_k}, \alpha_{i_{k+1}}\rangle=0$ for 
  some $1\leq k\leq r-1$, then  $(i_1, \ldots, i_{k-1}, i_{k+1}, i_k, i_{k+2},\ldots, i_r)$ is also a tuple corresponding to a reduced expression
  of $w$. 
  
 Two reduced expressions $\underline i$ and $\underline j$ of $w$ are said to differ only by commuting relations if $\underline j$ is obtained from $\underline i$ by a sequence of process as above.

 In this section we prove the following:
   \begin{theorem}\label{com}
  Let $w=s_{i_1}s_{i_2}\cdots s_{i_r}$ and $w=s_{j_1}s_{j_2}\cdots s_{j_r}$ be two reduced
  expressions $\underline i$ and $\underline j$ for $w$ which differ only by commuting relations. Then, 
 there is  a $B$-equivariant isomorphism  from $Z(w, \underline i)$ onto $Z(w, \underline j)$.
  \end{theorem}
  \begin{proof}
By the recursion, with out of loss of generality, we may assume that there is a $1\leq k\leq r-1$ such that $\langle \alpha_{i_{k}}, \alpha_{i_{k+1}} \rangle =0$ and  satisfying 
$$i_l=j_l ~ \mbox{for}  ~ l\neq k, ~ k+1 ;  ~ i_k=j_{k+1},  ~ i_{k+1}=j_k.$$   

Let $P$ be the parabolic subgroup of $G$ corresponding to the subset $\{\alpha_{i_{k}}, \alpha_{i_{k+1}}\}$ of  $S$. Since the simple roots  $\alpha_{i_{k}}$  and $\alpha_{i_{k+1}}$ are orthogonal, 
the product map $P_{\alpha_{i_{k}}}\times P_{\alpha_{i_{k+1}}}\to P$ is surjective.

Let $$X:=(P_{\alpha_{i_1}}\times \cdots \times P_{\alpha_{i_{k-1}}}\times P \times P_{\alpha_{i_{k+2}}}\times\cdots \times P_{\alpha_{i_{r}}})/B^{r-1}$$ be the quotient variety where the action of $B^{r-1}$ 
on $$P_{\alpha_{i_1}}\times \cdots \times P_{\alpha_{i_{k-1}}}\times P \times P_{\alpha_{i_{k+2}}}\times\cdots \times P_{\alpha_{i_{r}}}$$ is given by $(p_1,\ldots, p_{k-1}, p, p_{k+2}, \ldots, p_{r} )\cdot (b_1, b_2,\ldots, b_{k-1}, b_{k+1}, \ldots, b_{r})
=(p_1b_1, b_{1}^{-1}p_2b_2,\ldots, b_{k-1}^{-1}pb_{k+1}, b_{k+1}^{-1}p_{k+2}b_{k+2},\ldots, b_{r-1}^{-1}p_{r}b_{r}),$ $p_{j}\in P_{\alpha_{i_{j}}}$, $p\in P$ and $b_{j}\in B$. 

Note that $X$ is a smooth projective variety.
Now consider the map 

$ P_{\alpha_{i_1}}\times \cdots \times P_{\alpha_{i_{k-1}}} \times P_{\alpha_{i_{k}}} \times P_{\alpha_{i_{k+1}}}\times\cdots\times  
P_{\alpha_{i_{r}}}\longrightarrow P_{\alpha_{i_1}}\times \cdots \times P_{\alpha_{i_{k-1}}}\times P \times P_{\alpha_{i_{k+2}}}\times\cdots \times P_{\alpha_{i_{r}}}$

given by
$$(p_1,\ldots, p_{k-1},p_{k}, p_{k+1}, p_{k+2},\ldots, p_{r})\mapsto (p_1,\ldots, p_{k-1},p_{k}p_{k+1}, p_{k+2},
\ldots,p_{r}).$$
This induces a birational surjective $B$-equivariant morphism 
$f:Z(w, \underline i) \to X.$

\underline{Claim}: $f$ is injective.

We denote by $[(p_1,p_2,\ldots, p_{k}p_{k+1}, p_{k+2}, \ldots , p_{r})]$ be the point in $X$ corresponding to 
$(p_1,p_2,\ldots, p_{k-1}, p_{k}p_{k+1}, p_{k+2}, \ldots, p_{r})\in P_{\alpha_{i_1}}\times
\cdots \times P_{\alpha_{i_{k-1}}}\times P \times P_{\alpha_{i_{k+2}}}\times\cdots \times P_{\alpha_{i_{r}}}$.

If $[(p_1,p_2,\ldots,p_{k-1}, p_{k}p_{k+1}, \ldots, p_{r})]=[(p'_1,p'_2,\ldots, p'_{k-1}, p'_{k}p'_{k+1}, \ldots, p_{r})]$, then there exists $(b_1, b_2, \ldots , b_{r-2}, b_{r})$
such that $p_1=p_1'b_1$, $p_l=b_{l-1}^{-1}p_l'b_{l}$ for all $l\neq k,k+1$ and $p_{k}p_{k+1}=b_{k-1}^{-1}p'_{k}p'_{k+1}b_{k+1}$.
Hence we have  $(b_{k-1}^{-1}p'_{k})^{-1}p_{k}=p_{k+1}'b_{k+1}p_{k+1}^{-1}(~say~ =b_k) \in P_{\alpha_{i_k}}\cap P_{\alpha_{i_{k+1}}}=B$.
Therefore $p_{k}=b_{k-1}^{-1}p_{k}'b_{k}$ and $p_{k+1}=b_{k}^{-1}p_{k+1}'b_{k+1}$.
Thus, we have $(p_1,\ldots, p_{k-1},p_{k}, p_{k+1}, p_{k+2},\ldots, p_{r})$ and $(p'_1,\ldots, p'_{k-1},p'_{k}, p'_{k+1}, p'_{k+2},\ldots, p'_{r})$
represents the same element in $Z(w, \underline i)$.
Hence $f$ is injective.

Since $X$ is normal and $f$ is bijective birational, by Zariski main theorem \cite[p. 85, Theorem 5.2.8]{Springer}, we conclude that $f$ is an isomorphism. 
 Similarly, we see that there is a $B$-equivariant isomorphism from $Z(w, \underline j)$ onto $X$.
Thus, there is a $B$-equivariant isomorphism from $Z(w, \underline i)$  onto $Z(w, \underline j).$
\end{proof}

 \begin{corollary}
 \label{com1}
 Let $w=s_{i_1}s_{i_2}\cdots s_{i_r}$ and $w=s_{j_1}s_{j_2}\cdots s_{j_r}$ be two reduced
  expressions $\underline i$ and $\underline j$ for $w$ which differ only by commuting relations. Then,
 we have 
 
\begin{enumerate}
   \item  The automorphism groups $Aut^0(Z(w, \underline i))$ and $Aut^0(Z(w, \underline j))$ are isomorphic.
   \item The first cohomology groups
  $H^1(Z(w, \underline i), T_{(w, \underline i)})$ and $H^1(Z(w, \underline j), T_{(w, \underline j)})$
are isomorphic.
  \end{enumerate}
 \end{corollary}
 \begin{proof}
  The proof follows from Proposition \ref{com}.
 \end{proof}
 
 Recall from \cite{Stembridge1} and \cite{Stembridge2} the definition of a fully commutative element of a Coxeter group. 
 An element $w\in W$ is said to be fully commutative 
 if it has the property
 that any reduced expression for $w$ can be obtained from any other by using only the Coxeter relations that involve commuting generators.
 These elements are characterized in \cite{Stembridge1} and \cite{Stembridge2}.
 
 Then we have 
  \begin{corollary}
  Let $w$ be a fully commutative element in $W$. Then, $Z(w, \underline i)$ is independent of the choice of the reduced expression
  $\underline i$ of $w$.
 \end{corollary}
 
 \section{Reduced expressions of some elements of  $W$  in  type $C_{n}$}
 
 Now onwards we will assume that $G=PSp(2n, \mathbb C)$ ($n\geq 3$).
 First note that if $n\geq 3$, then the highest short root $\beta_0$ is $\omega_2$ and  $w_0=-id$.
We  recall the following proposition from \cite{Zele}(see \cite[Proposition 1.3]{Zele}). 
   We use the  notation as in \cite{Zele}.
   \begin{proposition} \label{Z}   
   Let $c \in W$ be a Coxeter element, let $\omega_i$ be a fundamental weight corresponding 
   to the simple root $\alpha_i$. Then, there exists a least positive integer 
   $h(i, c)$ such that $c^{h(i, c)}(\omega_i)=w_0(\omega_i)$.
    \end{proposition}
   Now we can deduce the following:
 \begin{lemma}\label{lemma: 1} Let $c \in W$ be a Coxeter element. Then, we have
 \begin{enumerate}
  \item $w_0=c^n$. 
  \item For any sequence $\underline i^r$ $(1\leq r\leq n)$ of reduced expressions of $c$; 
  the sequence $\underline i=(\underline i^1, \underline i^2,\ldots, \underline i^n)$ is a reduced expression of $w_0$.
\end{enumerate}
  \end{lemma}
 
 \begin{proof} Proof of (1): Let $\eta:S\longrightarrow S$ be the involution of $S$ defined by 
 $i\longmapsto i^*$, where $i^*$ is given by $\omega_{i^*}=-w_0(\omega_i)$.
 Since $G$ is of type $C_n$, $w_0=-id$ and hence $\omega_{i^*}=\omega_i$ for every $i$. 
 Therefore, we have $i=i^*$ for every $i$.
 Let $h$ be the Coxeter number. By \cite[Proposition 1.7]{Zele}, we have 
 $h(i,c)+h(i^*,c)=h$.
 Since $h=\frac{2|R^+|}{n}$ (see \cite[Proposition 3.18]{Humreflection}) and $i=i^*$, we have $h(i,c)=n$.

 By Proposition \ref{Z}, we have $c^n(\omega_i)=-\omega_i$ for all $1\leq i \leq n$.
Since $\{\omega_i : 1\leq i\leq n\}$ forms a  $\mathbb R$-basis of $X(T)\otimes \mathbb R$, it follows that $c^n=-id$.
Hence, we have $w_0=c^n$. 

The assertion (2) follows from the fact that $l(c)=n$ and $l(w_0)=|R^+|=n^2$ (see \cite[p.66, Table 1 ]{Hum1}).
  \end{proof}
  
 {\it We note that there are reduced expressions of $w_{0}$  in $C_n$ which are neither of the form
as in Lemma 4.2(2) nor differing from such a form by commuting relations.  For example, we can take 
 $n=3$. Then, the reduced expression $w_{0}=s_{2}s_{3}s_{2}s_{3}s_{1}s_{2}s_{3}s_{2}s_{1}$ is 
  one such. }

    \begin{lemma}\label{reduced}
   Let $n \geq a_1>a_2>a_3>\ldots >a_{r}\geq 1$ be a decreasing sequence of integers. Then,
   $$w=(\prod_{j=a_1}^n s_j)(\prod_{j=a_2}^n s_j)(\prod_{j=a_3}^{n} s_j)\cdots (\prod_{j=a_{r-1}}^{n} s_j)(\prod_{j=a_r}^{n-1} s_j)$$ is a 
   reduced expression of $w$.
  \end{lemma}
\begin{proof}
 If $r=1$, then $w=\prod\limits_{j=a_1}^{n-1} s_j$ and clearly it is a reduced expression of $w$.
  Let $$w_1=(\prod_{j=a_1}^n s_j)(\prod_{j=a_2}^{n} s_j)\cdots (\prod_{j=a_{r-2}}^{n} s_j)(\prod_{j=a_{r-1}}^{n-1} s_j) 
 ~~\mbox{and}~~ w_2=s_n(\prod_{j=a_r}^{n-1} s_j).$$ Then, we have  $w=w_1w_2$. 
  Now, we will prove that the expression $$(\prod_{j=a_1}^n s_j)(\prod_{j=a_2}^n s_j)(\prod_{j=a_3}^{n} s_j)\cdots (\prod_{j=a_{r-1}}^{n} s_j)(\prod_{j=a_r}^{n-1} s_j)$$ of $w_1w_2$ is reduced.

By induction on $r$, $w_1=(\prod\limits_{j=a_1}^n s_j)(\prod\limits_{j=a_2}^{n} s_j)\cdots (\prod\limits_{j=a_{r-1}}^{n-1} s_j)$ is a reduced expression. 
Note that, 
since $G$ is of type $C_n$, 
$$R^+(w_2^{-1})=\{\alpha_n, \sum\limits _{j=a_r}^{n}\alpha_j \}\cup \{\sum\limits_{j=a_r}^{m}\alpha_{j}: a_r\leq m\leq n-2\}.$$
%$R^-(w_1)=\{\beta\in R^+ : w_1(\beta)<0\}$ (\Leftrightarrow w_1s_{\beta}<w_1). Note that R^+(w_2)\cap R^-(w_1)=\emptyset. 
Since $n \geq a_1>a_2>a_3>\ldots >a_{r}\geq 1$,
it follows that $$R^+(w_1)\cap R^+(w_2^{-1})=\emptyset .$$ Thus we have $l(w)=l(w_1)+l(w_2)$. 
Hence  $$w=(\prod_{j=a_1}^n s_j)(\prod_{j=a_2}^n s_j)(\prod_{j=a_3}^{n} s_j)\cdots (\prod_{j=a_{r-1}}^{n} s_j)(\prod_{j=a_r}^{n-1} s_j)$$ is a 
   reduced expression of $w$.
  This completes the proof of the lemma. 
 \end{proof}

Let $c$ be a Coxeter element in $W$.
We take a reduced expression  $c=[a_1, n][a_2, a_{1}-1]\cdots [a_k, a_{k-1}-1]$, 
where $[i, j]=s_is_{i+1}\cdots s_j$ for $i\leq j$
and $n\geq a_1>a_2>\ldots >a_k=1$.

Then we have the following.
\begin{lemma} \label{coxeter}{\ }

\begin{enumerate}
 \item 
 For all $1\leq i\leq k-1$, $$c^i=(\prod_{l_1=1}^{i}[a_{l_1}, n])(\prod_{l_2=i+1}^{k}[a_{l_2}, a_{l_2-i}-1])(\prod_{l_3=1}^{i-1}[a_{k}, a_{k-i+l_3}-1]).$$
%(1) $c^i=[a_1, n][a_2, n]\cdots [a_i, n][a_{i+1}, a_1-1][a_{i+2}, a_2-1]\cdots [a_{k-2}, a_{k-(i+2)}-1][a_{k-1}, a_{k-(i+1)}-1][a_k, a_{k-i}-1][a_k, a_{k-i+1}-1]\cdots [a_k, a_{k-1}-1]$ 
 
 \item For all $k\leq j \leq n$,
 $$c^j=(\prod_{l_1=1}^{k-1}[a_{l_1}, n])([a_k, n]^{j+1-k})(\prod_{l_2=1}^{k-1}[a_{k}, a_{l_2}-1]).$$
\item The expressions of $c^i$ for $1\leq i\leq n$ as in $(1)$ and $(2)$ are reduced. 
\end{enumerate}
\end{lemma}
 \begin{proof}
Proof of (1) is by induction on $i$ and by commuting relation in the Weyl group.
Assume that for a fixed $1\leq i\leq k-2$, 
$$c^i=(\prod_{l_1=1}^{i}[a_{l_1}, n])(\prod_{l_2=i+1}^{k}[a_{l_2}, a_{l_2-i}-1])(\prod_{l_3=1}^{i-1}[a_{k}, a_{k-i+l_3}-1]).$$
Now we will prove, $$c^{i+1}=(\prod_{l_1=1}^{i+1}[a_{l_1}, n])(\prod_{l_2=i+2}^{k}[a_{l_2}, a_{l_2-(i+1)}-1])(\prod_{l_3=1}^{i}[a_{k}, a_{k-(i+1)+l_3}-1]).$$

Let $w_1:=\prod\limits_{l_1=1}^{i}[a_{l_1}, n], w_2:= \prod\limits_{l_2=i+1}^{k}[a_{l_2}, a_{l_2-i}-1]$ and $w_3:=\prod\limits_{l_3=1}^{i-1}[a_{k}, a_{k-i+l_3}-1]$.
Let $w_1':=\prod\limits_{l_1=1}^{i+1}[a_{l_1}, n], w_2':= \prod\limits_{l_2=i+2}^{k}[a_{l_2}, a_{l_2-(i+1)}-1]$ and 
$w_3':=\prod\limits_{l_3=1}^{i}[a_{k}, a_{k-(i+1)+l_3}-1]$.

Let $w_{2, 1}:=[a_{i+1} , a_1 -1]$, and let $v_1:=[a_1,  n]$.  Therefore, we have 
$w_{1}'=w_{1}w_{2, 1}v_{1}$.  Further, let  $w_{2, j}:=[a_{i+j} , a_j -1]$   ( $2\leq j \leq  k-i $ ).
Also, let  $w_{3, j}:=[a_{k} , a_j -1]$   ( $2\leq j \leq  k-i $ ), and  $v_j:=[a_j,  a_{j-1} -1]$  ( $2\leq j \leq  k$ ).

Now look at $c^i\cdot c=w_1w_2w_3v_1\cdots v_k$.  It is easy to derive the  following commuting relations  : 

\begin{itemize}
\item $v_1$ commutes with  each $w_{2, j}$ and $w_{3, j}$  ( $2\leq j \leq  k-i $ ).
\item For each $2\leq j \leq k-i$, $v_j$ commutes with  $w_{2, l}$  for each $j+1\leq l \leq  k-i$
and $w_{3, l}$  for each ( $k+1-i\leq l \leq  k-1 $ ).
\end{itemize}

Using these relations, we see that $$c^{i}c=( w_{1}w_{2, 1}v_{1} )  ( \prod\limits_{j=2}^{k-i}w_{2, j}v_{j})
( \prod\limits_{j=k+1-i}^{k-1}w_{3, j}v_{j})v_{k}.$$

Now, the statement (1) can be  obtained from the following facts:

\begin{itemize}
\item $w_{1}'=w_{1}w_{2, 1}v_{1}$,
\item $w_{2, j}v_{j}=[a_{i+j}, a_{j-1} -1]$ for each $2\leq j \leq k-i$ and 
\item $w_{3, j}v_{j}= [a_{k}, a_{j-1}-1]$ for each $k+1-i \leq j \leq k-1$.
\end{itemize}

Proof of $(2)$ is similar to the proof of (1).
Proof of $(3)$ follows from the fact that $l(c^i)=in$ for all $1\leq i\leq n$ and Lemma \ref{reduced}.
\end{proof}

\begin{example}
 Let $G$ is of type $C_3$. In this case we have the following four Coxeter elements;
 \begin{itemize}
  \item $c_1=s_1s_2s_3$; with $k=1; a_1=1.$
  \item $c_2=s_3s_1s_2$; with $k=2; a_1=3, a_2=1.$
  \item $c_3=s_2s_3s_1$; with $k=2; a_1=2,a_2=1.$
  \item $c_4=s_3s_2s_1$; with $k=3; a_1=3, a_2=2, a_3=1$.
  \end{itemize} Note that $h(i, c_j)=3$ for all $1\leq i \leq 3$ and $1\leq j \leq 4$. 
  
  The reduced expressions of $c_1^j$  ( $1\leq j \leq 3$ ) appearing in Lemma 4.4 are:
 \begin{itemize}
  \item $c_1=s_1s_2s_3.$
  \item $c_1^2=(s_1s_2s_3)(s_1s_2s_3).$
  \item $c_1^3=(s_1s_2s_3)(s_1s_2s_3)(s_1s_2s_3).$
  \end{itemize}

The reduced  expressions of $c_2^j$  ( $1\leq j \leq 3$ ) appearing in 
 Lemma 4.4  are :
\begin{itemize}
  \item $c_2=s_3s_1s_2$
  \item $c_2^2=s_3(s_1s_2) (s_3s_1s_2)$
  \item $c_2^3=s_3(s_1s_2s_3)(s_1s_2s_3)(s_1s_2)$
  
   \end{itemize}

\end{example}

 \section{Cohomology module $H^0$ of the relative tangent bundle} 
In this section we describe the weights of $H^0$ of the relative tangent bundle.

For a $B$-module $V$ and a character $\mu\in X(T)$, we denote by $V_{\mu}$ the space of all 
vectors $v$ in 
$V$ such that $t\cdot v=\mu(t)v$ for all $t \in T$.   
Let $R_s$ (respectively, $R_s^-$) be the set of all short roots (respectively, negative short roots).

\begin{lemma}\label{lambda}
 Let $w\in W$ and  let $V$ be a $B$-module such that $V_{\mu}=0$ unless $\mu\in R^-_s\setminus (-S)$. Then  $H^0(w, V)_{\mu}=0$ unless $\mu\in R^-_s\setminus (-S)$.
 \end{lemma}
\begin{proof}
Fix $\alpha\in S$ and let $\lambda_0=-\alpha$. Then by \cite[Lemma 4.1 (1)]{CKP}, 
$H^0(w, V)_{\mu}= 0$ unless $\mu\in R^-_s$ and  $\mu<-\alpha$.
Since $\alpha\in S$ is arbitrary, we have $H^0(w, V)_{\mu}=0$ unless $\mu\in R^-_s\setminus (-S)$.
\end{proof} 

\begin{lemma}\label{H0}
Fix $1\leq j\leq n-2$.
 \begin{enumerate}
  \item 
  $H^0(s_js_{j+1}\cdots s_{n-1}, \alpha_{n-1})_{\mu}\neq 0$ if and only if  either  $\mu=0$, or $\mu=s_ts_{t+1}\cdots s_{n-1}(\alpha_{n-1})$ for some $j\leq t\leq n-1$.
  \item $H^0(s_ns_{j}s_{j+1}\cdots s_{n-1}, \alpha_{n-1})_{\mu}\neq 0$ if and only if  either  
  $\mu=0$, or $\mu=vs_ts_{t+1}\cdots s_{n-1}(\alpha_{n-1})$ for some $j\leq t\leq n-1$ and 
  $v\in \{id, s_n\}$.
  \item Let $a_1, a_2$ be two integers such that $n\geq a_1 > a_2 \geq 1$.
  Then, $H^0(s_{a_1}\cdots s_{n}s_{a_2}\cdots s_{n-1}, \alpha_{n-1})_{\mu}\neq 0$ if and only if 
  $\mu=s_{t}\cdots s_{n-1}s_{n}s_{j}\cdots s_{n-1}(\alpha_{n-1})$ for some  $a_{2}\leq j \leq n-1$ and $a_1\leq t \leq n$  such that
  $t>j$.
 \end{enumerate}
\end{lemma}
\begin{proof}
Proof of (1):
 If $j=n-2$, then by SES we have $$H^0(s_{n-2}s_{n-1}, \alpha_{n-1})= H^0(s_{n-2}, H^0(s_{n-1}, \alpha_{n-1})).$$
 Note that $$H^0(s_{n-1}, \alpha_{n-1})=sl_{2, \alpha_{n-1}}=\mathbb C_{\alpha_{n-1}}
 \oplus \mathbb Ch_{\alpha_{n-1}}\oplus \mathbb C_{-\alpha_{n-1}},$$ where $h_{\alpha_{n-1}}$ is a zero weight vector in $sl_{2, \alpha_{n-1}}$.
 Since $\langle \alpha_{n-1}, \alpha_{n-2} \rangle=-1$,
 by Lemma \ref{lemma2.2}, we see that
 $$H^0(s_{n-2}s_{n-1}, \alpha_{n-1})= \mathbb Ch_{\alpha_{n-1}}\oplus \mathbb C_{-\alpha_{n-1}} \oplus \mathbb C_{-(\alpha_{n-1}+\alpha_{n-2})}.$$
 
 We prove the result by recursion.
  Now assume that the statement holds for $H^0(s_{l+1}\cdots s_{n-1}, \alpha_{n-1})$ and $j \leq l \leq n-3$.
 Then, for any $\mu\in X(T)$ such that $H^0(s_{l+1}\cdots s_{n-1}, \alpha_{n-1})_{\mu}\neq 0$, we have
 $\langle \mu, \alpha_j \rangle \in \{0, 1\}$.
 Thus, by Lemma \ref{dem1} and by using $SES$, we see that 
 $H^0(s_ls_{l+1}\cdots s_{n-1}, \alpha_{n-1})_{\mu}\neq 0$ if and only if  
 either  $\mu=0$, or $\mu=s_ts_{t+1}\cdots s_{n-1}(\alpha_{n-1})$ for some $l\leq t\leq n-1$.
   Therefore, proof of (1) follows by recursion.

   Proof of (2): By $SES$, we have $$H^0(s_ns_{j}s_{j+1}\cdots s_{n-1}, \alpha_{n-1})=H^0(s_n, H^0(s_{j}s_{j+1}\cdots s_{n-1}, \alpha_{n-1})).$$ 
   Note that $\langle s_t\cdots s_{n-1}(\alpha_{n-1}), \alpha_{n} \rangle=1$ for all $j\leq t \leq n-1$. 
 Now the proof of (2) follows from Lemma \ref{dem1} and (1).
  
  Proof of (3) is similar to the proof of (1) and (2).
  \end{proof}
  
Now onwards we fix the following notation:

Let $c$ be a Coxeter element in $W$.
We take a reduced expression  $c=[a_1, n][a_2, a_{1}-1]\cdots [a_k, a_{k-1}-1]$, 
where $[i, j]=s_is_{i+1}\cdots s_j$ for $i\leq j$  
and $n\geq a_1>a_2>\ldots >a_k=1$.

Fix $1\leq r \leq k$.
Let $n \geq a_1>a_2>a_3>\ldots >a_{r}\geq 1$ be a decreasing sequence of integers.\\
Let $$w_r=(\prod\limits_{j=a_1}^n s_j)(\prod\limits_{j=a_2}^n s_j)(\prod\limits_{j=a_3}^{n} s_j)\cdots (\prod\limits_{j=a_{r-1}}^{n} s_j)(\prod\limits_{j=a_r}^{n} s_j)$$ and
let $$\tau_r=(\prod\limits_{j=a_1}^n s_j)(\prod\limits_{j=a_2}^n s_j)(\prod\limits_{j=a_3}^{n} s_j)\cdots (\prod\limits_{j=a_{r-1}}^{n} s_j)(\prod\limits_{j=a_r}^{n-1} s_j).$$
Note that $l(w_r)=l(\tau_r)+1$ and $c\tau_k=w_k[a_k, n-1]$.

The following lemma describes the weights of $H^0$ of the relative tangent bundles.
\begin{lemma}\label{Weights}

Let $3\leq r\leq k$ . Then, we have
\begin{enumerate}
 \item
$H^0(\tau_r, \alpha_{n-1})_{\mu}\neq 0$ if and only if $\mu= (\prod\limits_{j=t_1}^n s_j)(\prod\limits_{j=t_2}^{n-1} s_j)(\alpha_{n-1})$ 
for some $a_{r}\leq t_2< t_1$ and $a_{r-1}\leq t_1\leq a_{r-2}-1$. 
In such a case, $dim(H^0(\tau_r, \alpha_{n-1})_{\mu})=1$.
\item $H^0(\tau_r, \alpha_{n-1})$ is a cyclic $B$-module generated by a weight vector of weight
$\mu=(\prod\limits_{j=a_{r-2}-1}^{n}s_j)(\prod\limits_{j=a_{r-2}-2}^{n}s_j)(\alpha_{n-1})$.
\end{enumerate} 
\end{lemma}
\begin{proof}
Note that there exists a $w_1\in W$ such that $w_1(\alpha_{n-1})=\beta_0$ and $l(\tau_rw_1^{-1})=l(\tau_r)+l(w_1)$, where $\beta_0$
is the highest short root. As a consequence, we have 
 $H^0(\tau_r, \alpha_{n-1})\subset H^0(\tau_r w_1^{-1}, \beta_0)$. Hence it follows 
 that for every non zero weight $\mu$ of $H^0(\tau_r, \alpha_{n-1})$, we have $dim(H^0(\tau_r, \alpha_{n-1})_{\mu})=1$ (see the proof of \cite[Lemma 4.4 ]{CKP}).

Proof of (1):
Let $u_1=s_{a_r}s_{a_r+1}\cdots s_{n-1}$, $u_2=s_{a_{r-1}}\cdots s_n$ and $u_3=s_{a_{r-2}}\cdots s_n$. 
 Let $w'=u_2u_1w_1^{-1}$.

\underline{Step 1:} We prove  $H^0(u_2u_1, \alpha_{n-1})_{\mu}\neq 0$ if and only if $\mu \in R^-_s$ and 
$\mu=v_1(\alpha_{n-1})$, where $ s_{n}s_{n-1} \leq v_1 \leq u_2u_1 $.

\underline{Claim:} $H^0(u_2u_1, \alpha_{n-1})$ is the $B$-submodule of $H^0(w', \beta_0)$ generated by $\mathbb C_{s_ns_{n-1}(\alpha_{n-1})}$.

Consider the following commutative  diagram of $B$-modules:
\begin{center}  
$\xymatrix{ H^0(s_{n-1}s_nu_1, \alpha_{n-1}) \ar[dr]^{ev_3} \ar[rr]^{ev_1} && H^0(s_nu_1, \alpha_{n-1}) \ar[dl]^{ev_2}\\ 
  & H^0(u_1, \alpha_{n-1})}$ 
\end{center}
By Lemma \ref{H0}(1),  if $H^0(u_1, \alpha_{n-1})_{\mu}\neq 0$, then $\mu=0$ or 
$\mu=s_js_{j+1}\cdots s_{n-1}(\alpha_{n-1})$ for some $a_r\leq j \leq n-1$.

First, observe that $\mathbb Ch_{\alpha_{n-1}}\oplus \mathbb C_{-\alpha_{n-1}}$ is an 
indecomposable $B_{\alpha_{n-1}}$-module (see \cite[p. 11]{CKP} and \cite[p. 8]{Ka4}).
By Lemma \ref{lemma2.4}, we have 
$$\mathbb Ch_{\alpha_{n-1}}\oplus \mathbb C_{-\alpha_{n-1}}=V\otimes \mathbb C_{-\omega_{n-1}},$$
where $V$ is the 2-dimensional irreducible representation of $\widetilde{L}_{\alpha_{n-1}}$.

 Note that  $\mathbb Ch_{\alpha_{n-1}}\oplus \mathbb C_{-\alpha_{n-1}}$ is an
 indecomposable $B_{\alpha_{n-1}}$-summand of $H^0(s_nu_1, \alpha_{n-1})$. By Lemma \ref{lemma2.3}, we have 
 
\begin{equation*}
  \begin{aligned}
       H^0(s_{n-1}, \mathbb Ch_{\alpha_{n-1}}\oplus \mathbb C_{-\alpha_{n-1}})  &= H^0(s_{n-1}, V\otimes \mathbb C_{-\omega_{n-1}})     \\
               &= V\otimes H^0(s_{n-1}, \mathbb C_{-\omega_{n-1}}) \\
               &= 0 
                 \end{aligned}
\end{equation*}

%$$H^0(s_{n-1}, \mathbb Ch_{\alpha_{n-1}}\oplus \mathbb C_{-\alpha_{n-1}})= H^0(s_{n-1}, V\otimes \mathbb C_{-\omega_{n-1}})=V\otimes H^0(s_{n-1}, \mathbb C_{-\omega_{n-1}})=0.$$
Also, note that  for each $\mu$ such that $H^0(u_1, \alpha_{n-1})_{\mu}\neq 0$ with $\mu \notin \{0, -\alpha_{n-1}\}$,
$\mathbb C_{\mu}$ 
is an  indecomposable $B_{\alpha_{n-1}}$-
summand of $H^0(s_nu_1, \alpha_{n-1})$. 
Since $\langle s_t \ldots s_{n-1}(\alpha_{n-1}), \alpha_{n-1} \rangle =-1$ for all $j\leq t \leq n-2$, using Lemma \ref{H0}(1),
 we see that the evaluation map $ev_3$ is zero.

Now consider the following commutative  diagram of $B$-modules.
\begin{center}  
$\xymatrix{ H^0(u_2u_1, \alpha_{n-1}) \ar[dr]^{ev} \ar[rr]^{ev'} && H^0(s_{n-1}s_nu_1, \alpha_{n-1}) \ar[dl]^{ev_3}\\ 
 & H^0(u_1, \alpha_{n-1})}$ 
\end{center} 

By the above arguments and commutativity of the above diagram, we see that the evaluation map $ev$
is zero.
 Note that $\langle s_ns_{n-1}(\alpha_{n-1}), \alpha_{n-1} \rangle =0$ and $\langle s_ns_{n-1}(\alpha_{n-1}), \alpha_{j} \rangle \geq 0$
 for $1\leq j\leq n-2$.
 Hence $\mathbb C_{s_ns_{n-1}(\alpha_{n-1})}$ is an indecomposable $B_{\alpha_{n-1}}$-summand of $H^0(s_nu_1, \alpha_{n-1})$.
 Therefore, $\mathbb C_{s_ns_{n-1}(\alpha_{n-1})}$ is in the image of the evaluation map 
 $$ev': H^0(u_2u_1, \alpha_{n-1})\longrightarrow H^0(s_{n-1}s_nu_1, \alpha_{n-1}).$$
 Hence, $s_{n}s_{n-1}(\alpha_{n-1})$ is the highest weight of $H^0(u_2u_1, \alpha_{n-1})$.
 Since $\beta_0$ is dominant weight, the restriction map of
 $B$-modules $$H^0(G/B, \mathcal L(\beta_0))=H^0(w_0, \beta_0)\longrightarrow H^0(w', \beta_0)$$ is surjecive.
 
 Since $H^0(u_2u_1, \alpha_{n-1})$ is $B$-submodule of $H^0(w', \beta_0)$, the multiplicity of $s_{n}s_{n-1}(\alpha_{n-1})$ in $H^0(u_2u_1, \alpha_{n-1})$ is one.
Also, note that the lowest weight of $H^0(u_2u_1, \alpha_{n-1})$ is $u_2u_1(\alpha_{n-1})$. 
Let $\{e_{\beta}: \beta \in R\}\cup \{h_{\alpha}: \alpha \in S \} $ be the 
Chevalley basis  for $\mathfrak{g}$, where $\mathfrak g$ is the Lie algebra of $G$ (refer to \cite[Chapter VII]{Hum1})).  
Hence we conclude that $$H^0(u_2u_1, \alpha_{n-1})=\bigoplus\limits_{a_r\leq t_2< t_1\leq n-1} \mathbb Cx_{t_1, t_2},$$ where 
$x_{t_1, t_2}=e_{-\alpha_{t_1}}\cdots e_{-\alpha_{n-1}}e_{-\alpha_{t_2}}\cdots e_{-\alpha_{n-2}}(y)$ for some non zero vector 
$y\in \mathbb C_{s_ns_{n-1}(\alpha_{n-1})}$. 

Therefore, $H^0(u_2u_1, \alpha_{n-1})$ is the $B$-submodule of $H^0(w', \beta_0)$ generated by $\mathbb C_{s_ns_{n-1}(\alpha_{n-1})}$.
This proves the claim. 

Now Step 1 follows from the claim.

\underline{Step 2:} 
Fix $a_{r-2} \leq j \leq n$.
$H^0(s_j\cdots s_nu_2u_1, \alpha_{n-1})_{\mu}\neq 0$ if and only if $\mu \in R^-_s$ 
and  $\mu=v_1(\alpha_{n-1})$, where $s_{n-1}s_{n}s_{n-2}s_{n-1} \leq v_1 \leq u_2u_1 $
 and $\langle \mu, \alpha_t \rangle =0$ for all $j\leq t \leq n$.

 Let $\mu_j=s_j\cdots s_{n-1}(\alpha_{n-1}), a_r\leq j \leq n-1$.
 Since $ev_{3}$ is zero, $H^0(u_2u_1, \alpha_{n-1})_{\mu_j}=0$. Further, we have $H^0(u_2u_1, \alpha_{n-1})_{s_n(\mu_j)}=\mathbb C_{s_n(\mu_j)}$ (by Step 1).
Since $\langle s_n(\mu_j), \alpha_n \rangle =-1$, by Lemma \ref{lemma2.2}, $\mathbb C_{s_n(\mu_j)}$
 is not in the image of the evaluation map $$ev: H^0(s_nu_2u_1, \alpha_{n-1})\to H^0(u_2u_1, \alpha_{n-1}).$$
 In particular, we have $$H^0(s_nu_2u_1, \alpha_{n-1})_{\mu_j}=H^0(s_nu_2u_1, \alpha_{n-1})_{s_n(\mu_j)}=0.$$

 By recursion assume that $H^0(s_{j+1}\cdots s_nu_2u_1, \alpha_{n-1})_{\mu}\neq 0$ if and only if 
 $\mu=v_1(\alpha_{n-1})$ for some $s_{n-1}s_ns_{n-2}s_{n-1}\leq v_1 \leq u_2u_1$ and 
 $\langle \mu, \alpha_t \rangle =0$ for all
 $j+1\leq t \leq n$.
  Let $\mu$ be such that $H^0(s_{j+1}\cdots s_nu_2u_1, \alpha_{n-1})_{\mu}\neq 0$.
 If $\langle \mu, \alpha_j \rangle =0$, then by Lemma \ref{lemma2.2}, $\mathbb C_{\mu}$
 is in the image of the evaluation map $$H^0(s_js_{j+1}\cdots s_nu_2u_1, \alpha_{n-1}) \to 
 H^0(s_{j+1}\cdots s_nu_2u_1, \alpha_{n-1}).$$
 
  Otherwise, we have $\langle \mu, \alpha_j \rangle =-1$ ( as $\mu\in R^-_{s}\setminus \{-\alpha_{j}\})$. 
  Note that by recursion  
 $\langle \mu+\alpha_j, \alpha_{j+1} \rangle =-1$. Therefore again using recursion, we see that
 $H^0(s_{j+1}\cdots s_nu_2u_1, \alpha_{n-1})_{\mu+\alpha_j}=0$.
Hence $\mathbb C_{\mu}$ is an indecomposable $B_{\alpha_j}$-summand of $H^0(s_{j+1}\cdots s_nu_2u_1, \alpha_{n-1})$.
 Thus, $\mathbb C_{\mu}$ is not in the image of the evaluation map $$ev:H^0(s_{j}s_{j+1}\cdots s_nu_2u_1, \alpha_{n-1}) \longrightarrow
 H^0(s_{j+1}\cdots s_nu_2u_1, \alpha_{n-1}).$$
 In particular, we have $$H^0(s_{j}s_{j+1}\cdots s_nu_2u_1, \alpha_{n-1})_{\mu}=0.$$

 Hence, the assertion of Step 2 follows by recursion.
In particular, for $j=a_{r-2}$ we have  $H^0(u_3, H^0(u_2u_1, \alpha_{n-1}))_{\mu}\neq 0$ if and only if $\mu \in R^-_s$ and 
  $\mu=v_1(\alpha_{n-1})$, where $s_{n-1}s_{n}s_{n-2}s_{n-1} \leq v_1 \leq u_2u_1 $
 and $\langle \mu, \alpha_t \rangle =0$ for all $a_{r-2}\leq t \leq n$.

 Hence we see that   $H^0(u_3, H^0(u_2u_1, \alpha_{n-1}))_{\mu}\neq 0$ if and only if $\mu \in R^-_s$ and
 $\mu=v_1(\alpha_{n-1})$, 
where $ s_{a_{r-2}-1}\cdots s_{n-1}s_{n}s_{a_{r-2}-2}\cdots s_{n-2}s_{n-1} \leq v_1 \leq u_2u_1 $.
 
\underline{Step 3:} Let $J=S\setminus \{\alpha_{1}, \ldots, \alpha_{a_{r-2}}\}$.
Let $W_J$ be the subgroup of $W$ generated by $\{s_{\alpha_j}: j\in J\}$. 
Let $v'\in W_J$. We prove 
 $H^0(v'u_3u_2u_1, \alpha_{n-1})_{\mu}\neq 0$ if and only if $\mu \in R^-_s$ and $\mu=v_1(\alpha_{n-1})$, 
where $s_{a_{r-2}-1}\cdots s_{n-1}s_{n}s_{a_{r-2}-2}\cdots s_{n-2}s_{n-1} \leq v_1 \leq u_2u_1 $.
 
 By Step 2, we see that 
if $H^0(u_3u_2u_1, \alpha_{n-1})_{\mu}\neq 0$, then 
$\mu \in R^-_s$ 
and  $\mu=v_1(\alpha_{n-1})$, where $s_{n-1}s_{n}s_{n-2}s_{n-1} \leq v_1 \leq u_2u_1 $
 and $\langle \mu, \alpha_j \rangle=0$ for all $a_{r-2} \leq j\leq n$.
   Hence, by Lemma \ref{dem1}(1) and Lemma \ref{lemma2.2}(1) we conclude the proof of Step 3.
  
  From Step 3, we see that 
$H^0(\tau_r, \alpha_{n-1})_{\mu}\neq 0$ if and only if $\mu= (\prod\limits_{j=t_1}^n s_j)(\prod\limits_{j=t_2}^{n-1} s_j)(\alpha_{n-1})$, where $t_2<t_1$, $a_{r}\leq t_2\leq a_{r-2}-2$ and $a_{r-1}\leq t_1\leq a_{r-2}-1$.
This completes the proof of (1).
   
  Proof of (2) follows from $(1)$.
  \end{proof}

\begin{lemma}\label{H0u_1}
  $H^0(w_k[a_k, n-1], \alpha_{n-1})_{\mu}\neq 0$  if and only if 
 $\mu= (\prod\limits_{j=t_1}^n s_j)(\prod\limits_{j=t_2}^{n-1} s_j)(\alpha_{n-1})$ for some $a_{k}\leq t_2<t_1\leq a_{k-1}-1$.
In such a case, $dim(H^0(w_k[a_k, n-1], \alpha_{n-1})_{\mu})=1$.

\end{lemma}

 \begin{proof} 
     If $H^0(\tau_k, \alpha_{n-1})_{\mu}\neq 0$, then by Lemma \ref{Weights}, we have $H^0(\tau_k, \alpha_{n-1})_{\mu}=\mathbb C_{\mu}$.
   Let $\mu=-(\alpha_n+2(\sum\limits_{j=a_{k-1}}^{n-1}\alpha_j)+\sum\limits_{j=t_2}^{a_{k-1}-1}\alpha_j)$.
   Note that $\langle \mu, \alpha_{a_{k-1}-1} \rangle =1$. Then $H^0(s_{a_{k-1}-1}\tau_k, \alpha_{n-1})_{\mu_1}\neq 0$, where 
   $\mu_1=\mu-\alpha_{a_{k-1}-1}$. By recursion, we see that $H^0(s_i\cdots s_{a_{k-1}-1}\tau_k, \alpha_{n-1})_{\mu_i}\neq 0$ for some
   $\mu_i=\mu-\sum\limits_{j=a_{k-1}-i}^{a_{k-1}-1}\alpha_j$ with $1\leq i \leq a_{k-1}-(t_2+1)$.
   Hence, we conclude that 
 $H^0([a_k, a_{k-1}-1]\tau_k, \alpha_{n-1})_{\mu}\neq 0$  if and only if 
 $\mu= (\prod\limits_{j=t_1}^n s_j)(\prod\limits_{j=t_2}^{n-1} s_j)(\alpha_{n-1})$ for some  $a_{k}\leq t_2<t_1\leq a_{k-2}-1$.
 
 \underline{Claim:} $H^0([a_{k-1}, a_{k-2}-1][a_k, a_{k-1}-1]\tau_k, \alpha_{n-1})_{\mu}\neq 0$  if and only if 
 $\mu= (\prod\limits_{j=t_1}^n s_j)(\prod\limits_{j=t_2}^{n-1} s_j)(\alpha_{n-1})$ for some  $a_{k}\leq t_2<t_1\leq a_{k-1}-1$.
 
 Fix $m$ such that $a_{k-2}-1\leq m\leq a_{k-1}$. Let
 $\mu_{m}=(\alpha_n+2(\sum\limits_{j=m}^{n-1}\alpha_j)+\sum\limits_{j=t_2}^{m-1}\alpha_j)$.
 Note that $\langle \mu_{m}, \alpha_{m} \rangle =-1$. 
 Hence the claim follows from  SES and Lemma \ref{lemma2.2}.
 
 Let  $\mu= (\prod\limits_{j=t_1}^n s_j)(\prod\limits_{j=t_2}^{n-1} s_j)(\alpha_{n-1})$ for some  $a_{k}\leq t_2<t_1\leq a_{k-1}-1$.
Since $\langle \mu, \alpha_l \rangle =0$ for all $a_{k-2}\leq l \leq n$, we conclude that
$H^0(w_k[a_k, n-1], \alpha_{n-1})_{\mu}\neq 0$  if and only if 
 $\mu= (\prod\limits_{j=t_1}^n s_j)(\prod\limits_{j=t_2}^{n-1} s_j)(\alpha_{n-1})$ for some $a_{k}\leq t_2<t_1\leq a_{k-1}-1$.
\end{proof}

 Let $h(\alpha_n)\in \mathfrak h$ be the fundamental dominant coweight corresponding to $\alpha_n$. That is,
 $\alpha_{n}(h(\alpha_n))=1$ and $\alpha_i(h(\alpha_n))=0$ for $i\neq n$. These coweights are used in studing  the indecomposable $B_{\alpha}$-modules 
in \cite[p. 8, Lemma 3.3]{Ka4}).

Let $v_r'=s_{a_r}\cdots s_{n-1}$ and  $v_r=s_ns_{a_r}\cdots s_{n-1}$.

Let $V'=H^0(v_r', \mathbb Ch(\alpha_n)\oplus \mathbb C_{-\alpha_n})$ and 
$V=H^0(v_r, \mathbb Ch(\alpha_n)\oplus \mathbb C_{-\alpha_n})$.
\begin{lemma}\label{V_m}\

\begin{enumerate}
\item  $V'_{\mu}\neq 0$ if and only if $\mu$ is of the form \\
 (i) $\mu \in \{0,  -\alpha_n \}$ or \\
 (ii) $\mu=-(\sum\limits_{j=t}^{n}\alpha_j)~~\mbox{ for some} ~~ a_r\leq t<n $ or \\ 
 (iii) $\mu=-(\alpha_n+2(\sum\limits_{j=t}^{n-1}\alpha_j))~~ \mbox{ for some}~~ a_r\leq t\leq n-1$ or\\
(iv) $\mu=-(\alpha_n+2(\sum\limits_{j=t_1}^{n-1}\alpha_j)+\sum\limits_{j=t_2}^{t_1-1}\alpha_j) 
 ~~\mbox{for some }~~ a_r\leq t_2<t_1\leq n-1$.
 \item 
 $V_{\mu}\neq 0$ if and only if $\mu$ is of the form: \\
 (i) $\mu=-(\alpha_n+2(\sum\limits_{j=t}^{n-1}\alpha_j))$ for some $a_r\leq t \leq n-1$ or \\
 (ii) $\mu=-(\alpha_n+2(\sum\limits_{j=t_1}^{n-1}\alpha_j)+\sum\limits_{j=t_2}^{t_1-1}\alpha_j)$ for some $a_r\leq t_2<t_1\leq n-1$.

\end{enumerate}

 \end{lemma}
\begin{proof}
Proof of (1):
 Note that $H^0(s_{a_r}\cdots s_{n-1}, \mathbb Ch(\alpha_n)\oplus \mathbb C_{-\alpha_n})$ is a cyclic $B$-module generated by $h(\alpha_n)$.
 Therefore the weights are of the form (each with multiplicity one):
 
 $\{0, -\alpha_n\} \cup \{-(\sum\limits_{j=t}^{n}\alpha_j): a_r\leq t<n\} \cup 
 \{-(\alpha_n+2(\sum\limits_{j=t}^{n-1}\alpha_j)): a_r\leq t\leq n-1\} \cup
 \{-(\alpha_n+2(\sum\limits_{j=t_1}^{n-1}\alpha_j)+\sum\limits_{j=t_2}^{t_1-1}\alpha_j): a_r\leq t_2<t_1\leq n-1\}$.
 
 Proof of (2):
 By Lemma \ref{lemma2.4}, $\mathbb Ch(\alpha_n)\oplus \mathbb C_{-\alpha_n}=V_1\otimes\mathbb C_{-\omega_{n}}$, where $V_1$ is the $2$-dimensional irreducible
 $\tilde{L}_{\alpha_n}$-module.
 Since $\langle -(\sum\limits_{j=t}^{n}\alpha_j), \alpha_n \rangle =-1$ for all $a_r\leq t \leq n-1$, by SES we see that, the weights of 
 $H^0(s_ns_{a_r}\cdots s_n, \mathbb Ch(\alpha_n)\oplus \mathbb C_{-\alpha_n})$ are of the form (each with multiplicity one):
 
 $\{-(\alpha_n+2(\sum\limits_{j=t}^{n-1}\alpha_j)): a_r\leq t\leq n-1\}\cup 
\{-(\alpha_n+2(\sum\limits_{j=t_1}^{n-1}\alpha_j)+\sum\limits_{j=t_2}^{t_1-1}\alpha_j): a_r\leq t_2<t_1\leq n-1\}$.
\end{proof}

\begin{lemma}\label{H0ar-1}
 $H^0(s_{a_{r-1}}\cdots s_{n-1}, V)_{\mu}\neq 0$ if and if $\mu$ is of the form: \\ 
 (i) $\mu=-(\alpha_n+2(\sum\limits_{j=t}^{n-1}\alpha_j))$ for some $a_r\leq t \leq a_{r-1}-1$
or \\
(ii) $\mu=-(\alpha_n+2(\sum\limits_{j=t_1}^{n-1}\alpha_j)+\sum\limits^{t_1-1}_{j=t_2}\alpha_j)$ for some $a_r\leq t_2< t_1 \leq a_{r-1}-1$.
\end{lemma}
\begin{proof}
 We have  $H^0(s_{n-1}, V)=H^0(s_{n-1}, H^0(s_ns_{a_r}\cdots s_{n-1}, \mathbb Ch(\alpha_n)\oplus \mathbb C_{-\alpha_n})$.
Since $\langle -(\alpha_n+2\alpha_{n-1}), \alpha_{n-1} \rangle =-2$ and 
$\langle -(\alpha_n+2\alpha_{n-1}+\sum\limits_{j=t_2}^{n-2}\alpha_j), \alpha_{n-1} \rangle =-1$ for $a_r\leq t_2 \leq n-2$,
by Lemma \ref{lemma2.3} and Lemma \ref{V_m}, we see that the weights of $H^0(s_{n-1}, V)$ are of the form 

$\mu=-(\alpha_n+2(\sum\limits_{j=t}^{n-1}\alpha_j))$, where $a_r\leq t \leq n-2$
or\\
$\mu=-(\alpha_n+2(\sum\limits_{j=t_1}^{n-1}\alpha_j)+\sum\limits^{t_1-1}_{j=t_2}\alpha_j)$, where $a_r\leq t_2< t_1 \leq n-2$.

Now the proof of the lemma follows by similar arguments as above.
\end{proof}

 \begin{proposition}  \label{H0alphan} 
 Let $2 \leq r\leq k$. Then,
   $H^0(w_r, \alpha_n)_{\mu}\neq 0$ if and only if $\mu$ is of the form : \\ 
 (i) $\mu=-(\alpha_n+2(\sum\limits_{j=t}^{n-1}\alpha_j))$ for some $a_r\leq t \leq a_{r-1}-1$
or \\
(ii) $\mu=-(\alpha_n+2(\sum\limits_{j=t_1}^{n-1}\alpha_j)+\sum\limits^{t_1-1}_{j=t_2}\alpha_j)$ for some $a_r\leq t_2< t_1 \leq a_{r-1}-1$.
In such a case, $dim(H^0(w_r, \alpha_n)_{\mu})=1$.
  \end{proposition}
  \begin{proof}
 Let $u=(\prod\limits_{j=a_1}^n s_j)(\prod\limits_{j=a_2}^n s_j)(\prod\limits_{j=a_3}^{n} s_j)\cdots (\prod\limits_{j=a_{r-3}}^{n} s_j)
 (\prod\limits_{j=a_{r-2}}^{n} s_j)$
  and let $v=(\prod\limits_{j=a_{r-1}}^n s_j)(\prod\limits_{j=a_r}^{n} s_j)$.
 Observe that $w_r=uv$, $l(w_r)=l(uv)=l(u)+l(v)$ and by SES, we have $H^0(w_r, \alpha_n)=H^0(u, H^0(v, \alpha_n))$.
 
 Since $\langle \alpha_n, \alpha_{n-1}\rangle =-2$ and by SES,  $H^0(v, \alpha_n)=H^0(s_{a_{r-1}}\cdots s_{n-1}, V)$, where 
 $V=H^0(v_r, \mathbb Ch(\alpha_n)\oplus \mathbb C_{-\alpha_n})$.
 By Lemma \ref{H0ar-1}, $H^0(v, \alpha_n)_{\mu}\neq 0$ if and only if $\mu$ is of the form: \\ 
 (i) $\mu=-(\alpha_n+2(\sum\limits_{j=t}^{n-1}\alpha_j))$, where $a_r\leq t \leq a_{r-1}-1$
or \\
(ii) $\mu=-(\alpha_n+2(\sum\limits_{j=t_1}^{n-1}\alpha_j)+\sum\limits^{t_1-1}_{j=t_2}\alpha_j)$, where $a_r\leq t_2< t_1 \leq a_{r-1}-1$.

Note that for any such $\mu$, we have $\langle \mu, \alpha_j \rangle=0$ for all $a_{r-1} \leq j \leq n$.
Hence by Lemma \ref{dem1} and by SES, we conclude that $H^0(w_r, \alpha_n)_{\mu}\neq 0$ if and only if $\mu$ is as in the statement.
%of the form: \\ 
 %(i) $\mu=-(\alpha_n+2(\sum\limits_{j=r}^{n-1}\alpha_j))$, where $a_r\leq r \leq a_{r-1}-1$
%or \\
%(ii) $\mu=-(\alpha_n+2(\sum\limits_{j=t_1}^{n-1}\alpha_j)+\sum\limits_{j=t_1-1}^{t_2}\alpha_j)$, where $a_r\leq t_2< t_1 \leq a_{r-1}-1$.
 \end{proof}

\section{Cohomology module $H^1$ of the relative tangent bundle}

In this section we describe the weights of $H^1$ of a relative tangent bundle.

 Let $\mathfrak g'_{<\alpha_n}$
be the $B$-submodule of $\mathfrak g$ generated by $h(\alpha_n)$. Note that 
$$ \mathfrak g'_{<\alpha_n}=\mathbb C h(\alpha_n)\oplus \bigoplus \limits_{\beta\leq -\alpha_n} \mathfrak g_{\beta}.$$
\begin{lemma}\label{g'}
 Let $w\in W$. Let $V$ be a $B$-submodule of $\mathfrak g'_{<\alpha_n}$ such that either 
 $(\mathbb C h(\alpha_n)\oplus \mathfrak g_{-\alpha_n})\cap V=0$ or $V=\mathfrak g'_{<\alpha_n}$.
 Then, $H^1(w, V)_{\mu}= 0$ unless $\mu\in R^-_{s}\setminus (-S)$.
\end{lemma}
\begin{proof}
 The proof is by induction on $l(w)$. Assume that $l(w)=1$. Then $w=s_i$ for some $1\leq i\leq n$.
 If $H^1(s_i, V)_{\mu}\neq 0$,
 then there exists an indecomposable $B_{\alpha_i}$-direct summand $V_1$ of $V$ such that $H^1(s_i, V_1)\neq 0$.
 By Lemma \ref{lemma2.4}, we have $V_1=V'\otimes \mathbb C_{a\omega_i}$ for some irreducible 
 $\tilde{L}_{\alpha_i}$-module $V'$
 and an integer $a$.
 Since $H^1(s_i, V_1)\neq 0$ and $G$ is of type $C_n$, by Lemma \ref{lemma2.3}, we have $dim(V_1)=1$ and $a=-2$. Further, we have  
 $$H^1(s_i, V_1)=V'\otimes H^1(s_i, -2\omega_i)=V'\otimes H^0(s_i, -2\omega_i+\alpha_i).$$
 Therefore, $H^1(s_i, V_1)=\mathbb C_{\mu_1+\alpha_i}$, where $\mu_1$
is the lowest weight of $V_1$. By the hypothesis on $V$, we have $\mu_1=-\beta$ for some $\beta \in R^+\setminus S$ with $\langle \beta, \alpha_i \rangle=2$.
Therefore $\beta$ is a long root and $-\beta+\alpha_i\in R^-_s\setminus (-S)$.

Assume that $l(w)>1$. Choose $1\leq i\leq n$ such that $l(ws_i)=l(w)-1$. By \cite[Lemma 6.1]{Ka4}, we have the following exact
sequence of $B$-modules:
$$H^1(ws_i, H^0(s_i, V))\longrightarrow H^1(w, V)\longrightarrow H^0(ws_i, H^1(s_i, V))\hspace{1cm} (6.1.1)$$

\underline{Claim:} $H^0(s_i, V)\cap (\mathbb C h(\alpha_n)\oplus \mathfrak g_{-\alpha_n})=0$.

Assume that $i=n$. By Lemma \ref{lemma2.4}, we have $$\mathbb Ch(\alpha_n)\oplus \mathfrak g_{-\alpha_n}=V_2\otimes \mathbb C_{-\omega_n},$$
 $V_2$ is the 2-dimensional irreducible $\tilde L_{\alpha_n}$-module.
Hence we have $$H^0(s_n, \mathbb Ch(\alpha_n)\oplus \mathfrak g_{-\alpha_n})=V_2\otimes H^0(s_n, \mathbb C_{-\omega_n})=0.$$

Assume that $i\neq n$. If $V=\mathfrak g'_{<\alpha_n}$, then $H^0(s_i, V)=V$.
Otherwise, since  $H^0(s_i, V)\subset V$, we have  $$H^0(s_i, V)\cap (\mathbb Ch(\alpha_n)\oplus \mathfrak g_{-\alpha_n})=0.$$
By induction on $l(w)$, if $H^1(ws_i, H^0(s_i, V))_{\mu}\neq 0$, then $\mu\in R^-_{s}\setminus (-S)$.

By above (as in the case of $l(w)=1$), 
there is a descending sequence of $B$-modules:
$$H^1(s_i, V)\supsetneq V^1 \supsetneq V^2\supsetneq \ldots \supsetneq V^r=0$$
such that $V^i/V^{i+1}\simeq \mathbb C_{\beta_i}$ for some $\beta_i\in R^-_s\setminus (-S)$.
By Lemma \ref{lambda}, if $H^0(ws_i, H^1(s_i, V))_{\mu}\neq 0$, then $\mu\in R^-_s\setminus (-S)$.
Hence by the above exact sequence (6.1.1), we conclude that if $H^1(w, V)_{\mu}\neq 0$ then $\mu\in R^-_s\setminus (-S)$. This completes the proof.
\end{proof}

 \begin{proposition}\label{H1alphan}
  Let $ u\in W$ and  an integer $1\leq a\leq n-2$ be such that $l(us_as_{a+1}\cdots s_{n})=l(u)+(n+1-a)$. Let $w=us_as_{a+1}\cdots s_{n}$. 
  Then we have:
  \begin{enumerate}
   \item If $u=id$, then $H^1(w, \alpha_n)=0$.
   \item $H^j(ws_n, \alpha_n)=0$ for all $j\geq 0$.
   \item $H^1(w, \alpha_n)=H^1(ws_n, \mathbb Ch(\alpha_n)\oplus \mathbb C_{-\alpha_n})$.
   \item $H^1(w, \alpha_n)_{\mu}= 0$ unless $\mu\in R^-_s\setminus (-S)$.
 \end{enumerate}
 \end{proposition}
\begin{proof}
Proof of (2):
Since $l(ws_ns_{n-1})=l(ws_n)-1$, by Lemma \ref{lemma2.2}, we have
$$H^0(ws_n, \alpha_n)=0~~\mbox{and}~~$$ 
$$H^j(ws_n, \alpha_n)=H^{j-1}(ws_ns_{n-1}, H^0(s_{n-1}, s_{n-1}\cdot \alpha_n)) ~~\mbox{for all}~~ j\geq 1.$$
Since $s_{n-1}\cdot \alpha_n=\alpha_{n-1}+\alpha_n$ and  $H^0(s_{n-1}, \alpha_{n-1}+ \alpha_n)=\mathbb C_{\alpha_{n-1}+\alpha_n}$ , we have 
$$H^j(ws_n, \alpha_n)=H^{j-1}(ws_ns_{n-1},  \alpha_{n-1}+ \alpha_n)~~\mbox{for all}~~ j\geq 1.$$ 
Since $\langle \alpha_{n-1}+\alpha_n, \alpha_{n-2} \rangle =-1$ and $l(ws_ns_{n-1}s_{n-2})=l(ws_ns_{n-1})-1$, by Lemma \ref{lemma2.2}, we see that
 $H^{j-1}(ws_ns_{n-1}, \alpha_{n-1}+\alpha_n)=0$ for all $j\geq 1$.

Therefore, $H^j(ws_n, \alpha_n)=0$ for every $j\geq 0$. 
This completes the proof of (2).

Proof of (3):
 Consider the following short exact sequences of $B$-modules:
 $$0\longrightarrow K_1\longrightarrow H^0(s_n, \alpha_n)\longrightarrow \mathbb C_{\alpha_n}\longrightarrow 0 \hspace{1.5cm}(6.2.1)$$
 where $K_1$ is the kernel of the evaluation map $ev: H^0(s_n, \alpha_n)\longrightarrow \mathbb C_{\alpha_n}$.
 Note that $K_1=\mathbb Ch(\alpha_n)\oplus \mathbb C_{-\alpha_n}$.
 
By applying  $H^0(ws_n, -)$ to (6.2.1), we have the following long exact sequence of $B$-modules:

$\cdots \longrightarrow H^0(ws_n, \alpha_n) \longrightarrow H^1(ws_n, K_1)\longrightarrow H^1(ws_n, H^0(s_n, \alpha_n))\longrightarrow H^1(ws_n, \alpha_n)\longrightarrow \cdots$

By (2), we conclude that $$H^1(ws_n, K_1) = H^1(ws_n, H^0(s_n, \alpha_n))=H^1(w, \alpha_n).$$ 
 
 Proof of (4):
Let $K_2=\sum\limits_{\mu<-\alpha_n}{}\mathfrak g_{\mu}$. Clearly, $K_2$ is a $B$-submodule of $\mathfrak g'_{<\alpha_n}$ and 
$$\frac{\mathfrak g'_{<\alpha_n}}{K_2}\simeq K_1.$$
Then we have a following short exact sequence of $B$-modules:
 $$0\longrightarrow K_2\longrightarrow \mathfrak g'_{<\alpha_n} \longrightarrow K_1\longrightarrow 0 \hspace{1.5cm}(6.2.2)$$
 
 By applying $H^0(ws_n, -)$ to (6.2.2), we have the following long exact sequence of $B$-modules:
 
 $\cdots \longrightarrow H^1(ws_n, K_2)\longrightarrow H^1(ws_n, \mathfrak g'_{<\alpha_n}) \longrightarrow H^1(ws_n, K_1)\longrightarrow H^2(ws_n, K_2)\longrightarrow \cdots . $

 By \cite[Lemma 6.2]{Ka4}, we have $H^2(ws_n, K_2)=0$. 
Hence by Lemma \ref{g'} we see that  that if $H^1(ws_n, \mathfrak g'_{<\alpha_n} )_{\mu}\neq 0$, then $\mu\in R^-_s\setminus (-S)$.
Therefore, $H^1(ws_n, K_1)_{\mu}=0$ unless $\mu\in R^-_s\setminus (-S)$.
 
 By (3), we have $ H^1(w, \alpha_n)=H^1(ws_n, K_1).$ Hence we conclude  that $H^1(w, \alpha_n)_{\mu}=0$ unless $\mu\in R^-_{s}\setminus (-S)$.
 
 Proof of (1) is similar to the proof of (2).
 \end{proof}

\begin{lemma}\label{H^1(2)} Let $V=H^0(s_ns_{a_r}\cdots s_{n-1}, \mathbb Ch(\alpha_n)\oplus \mathbb C_{-\alpha_n})$. Then, 
 $H^1(s_{a_{r-1}}\cdots s_{n-1}, V)_{\mu}\neq 0$ if and if
  $\mu$ is of the form: \\ (i) $\mu=-(\sum\limits_{j=t}^{n}\alpha_j)$ for some $a_{r-1}\leq t \leq n-1$
or \\
(ii) $\mu=-(\alpha_n+2(\sum\limits_{j=t_1}^{n-1}\alpha_j)+\sum\limits^{t_1-1}_{j=t_2}\alpha_j)$ for some $a_{r-1}\leq t_2< t_1 \leq n-1$.
\end{lemma}
\begin{proof}
\underline{Case 1:} $a_{r-1}=n-1$.
By Lemma \ref{V_m}, it follows that  $V_{\mu}\neq 0$ if and only if $\mu$ is of the form: \\ 
(i) $\mu=-(\alpha_n+2(\sum\limits_{j=t}^{n-1}\alpha_j))$, where $a_r\leq t\leq n-1$ or \\
(ii) $\mu=-(\alpha_n+2(\sum\limits_{j=t_1}^{n-1}\alpha_j)+\sum\limits_{j=t_2}^{t_1-1}\alpha_j)$, where $a_r\leq t_2<t_1\leq n-1$.

If $t<n-1$(or $t_1<n-1$), then we have $\langle \mu, \alpha_{n-1} \rangle =0$. Hence by Lemma \ref{lemma2.2}, we see that 
$H^1(s_{n-1}, \mu)=0$. 

If $t=n-1$ in type (i), then   $\mu=-(\alpha_n+2\alpha_{n-1})$ and 
$\langle \mu, \alpha_{n-1} \rangle =-2$. Then by Lemma \ref{lemma2.2}, 
$$H^1(s_{n-1}, -(\alpha_n+2\alpha_{n-1}))=H^0(s_{n-1}, s_{n-1}\cdot (-(\alpha_n+2\alpha_{n-1}))).$$
Since $s_{n-1}\cdot (-(\alpha_n+2\alpha_{n-1}))=-(\alpha_n+\alpha_{n-1})$,
we see that $$H^1(s_{n-1}, -(\alpha+2\alpha_{n-1}))=\mathbb C_{-(\alpha_n+\alpha_{n-1})}.$$

If $t_1=n-1$ in type (ii), then $\mu=-(\alpha_n+2\alpha_{n-1}+\sum\limits_{j=t_2}^{n-2}\alpha_j)$ and $\langle \mu, \alpha_{n-1} \rangle =-1$. 
Then by Lemma \ref{lemma2.2}, 
we conclude that $H^1(s_{n-1}, \mu)=0$.

\underline{Case 2:}  $a_{r-1}\neq n-1$.
 Fix $a_{r-1}\leq i\leq n-2$. By recursion we assume that 
 
 $H^1(s_{i+1}\cdots s_{n-1}, V)_{\mu}\neq 0$ if and if $\mu=-(\sum\limits_{j=t}^{n}\alpha_j)$ for some  $i+1 \leq t \leq n-1$
or $\mu=-(\alpha_n+2(\sum\limits_{j=t_1}^{n-1}\alpha_j)+\sum\limits^{t_1-1}_{j=t_2}\alpha_j)$ for some $i+1\leq t_2< t_1 \leq n-1$.

By SES, we have the following short exact sequence of $B$-modules: 

$0\longrightarrow H^1(s_i, H^0(s_{i+1}\cdots s_{n-1}, V))\longrightarrow H^1(s_is_{i+1}\cdots s_{n-1}, V)\longrightarrow H^0(s_i, H^1(s_{i+1}\cdots s_{n-1}, V))\longrightarrow 0$.

By the above discussion, we see that $H^0(s_i, H^1(s_{i+1}\cdots s_{n-1}, V))_{\mu}\neq 0$ if and only if 
$\mu=-(\sum\limits_{j=t}^{n}\alpha_j)$, where $i \leq t \leq n-1$
or $\mu=-(\alpha_n+2(\sum\limits_{j=t_1}^{n-1}\alpha_j)+\sum\limits^{t_1-1}_{j=t_2}\alpha_j)$, where $i\leq t_2< t_1 \leq n-1$ and $t_1\geq i+2$.

Further, by Lemma \ref{H0ar-1}, we see that 
$H^1(s_i, H^0(s_{i+1}\cdots s_{n-1}, V))_{\mu}\neq 0$ if and only if $\mu=-(\alpha_n+2(\sum\limits_{j=i+1}^{n-1}\alpha_j)+ \alpha_i)$.
By the recursion, we conclude the proof of the lemma.
\end{proof}

Let $v_r=s_ns_{a_r}\cdots s_{n-1}$, $v_{r-1}=s_{a_{r-1}}\cdots s_{n-1}$ and $v_{r-2}=s_{a_{r-2}}\cdots s_{n-1}s_n$. Then we have
\begin{lemma}\label{H1(3)}\
 
 \begin{enumerate}

 \item $H^1(v_{r-1}v_rs_n, \alpha_n)_{\mu}\neq 0$ if and only if 
  $\mu$ is of the form: \\ (i) $\mu=-(\sum\limits_{j=t}^{n}\alpha_j)$ for some $a_{r-1}\leq t \leq n-1$
or \\
(ii) $\mu=-(\alpha_n+2(\sum\limits_{j=t_1}^{n-1}\alpha_j)+\sum\limits^{t_1-1}_{j=t_2}\alpha_j)$ for some $a_{r-1}\leq t_2< t_1 \leq n-1$.
 
 \item Let $w=v_{r-2}v_{r-1}v_rs_n$. Then, 
 $H^1(w, \alpha_n)_{\mu}\neq 0$ if and only if
 $\mu=-(\alpha_n+2(\sum\limits_{j=t_1}^{n-1}\alpha_j)+\sum\limits^{t_1-1}_{j=t_2}\alpha_j)$ for some $a_{r-1}\leq t_2< t_1 \leq a_{r-2}-1$.
 
 \end{enumerate}

 \end{lemma}
\begin{proof}
\underline{Step 1:} We prove $$H^1(v_{r-2}v_{r-1}, V)=H^0(v_{r-2}, H^1(v_{r-1}, V)),$$ where $V$ is as in Lemma \ref{H^1(2)}.
 
By SES we have the following short exact sequence of $B$-modules:

$0\longrightarrow H^1(s_n, H^0(v_{r-1}, V))\longrightarrow H^1(s_nv_{r-1}, V) \longrightarrow H^0(s_n, H^1(v_{r-1}, V))\longrightarrow 0.$

By Lemma \ref{H0ar-1}, if $H^0(v_{r-1}, V)_{\mu}\neq 0$, then we see that $\langle \mu, \alpha_i \rangle =0$ for all 
$a_{r-2}\leq i \leq n$. 
By Lemma \ref{lemma2.2}, we have 
$$H^1(s_n, H^0(v_{r-1}, V))=0~~ \mbox{and}$$  $$ H^0(s_n, H^0(v_{r-1}, V))=H^0(v_{r-1}, V).$$

Hence by the above short exact sequence, we see that   $$H^1(s_nv_{r-1}, V)= H^0(s_n, H^1(v_{r-1}, V)).$$

By recursion, we have  $$H^1(s_{i+1}\cdots s_nv_{r-1}, V)= H^0(s_{i+1}\cdots s_n, H^1(v_{r-1}, V)).$$
By SES, we have the following short exact sequence of $B$-modules:

$0\longrightarrow H^1(s_i, H^0(s_{i+1}\cdots s_nv_{r-1}, V))\longrightarrow H^1(s_i\cdots s_nv_{r-1}, V) \longrightarrow H^0(s_i, H^1(s_{i+1}\cdots s_n v_{r-1}, V))\longrightarrow 0.$

Since $H^0(s_{i+1}\cdots s_nv_{r-1}, V)= H^0(v_{r-1}, V)$ and 
$H^1(s_i, H^0(s_{i+1}\cdots s_nv_{r-1}, V)=0$, we have 
$$H^1(s_is_{i+1}\cdots s_nv_{r-1}, V)=H^0(s_is_{i+1}\cdots s_n, H^1(v_{r-1}, V)).$$
Therefore, we conclude that  $$H^1(v_{r-2}v_{r-1}, V)=H^0(v_{r-2}, H^1(v_{r-1}, V)).$$
This completes the proof of Step 1.

\underline{Step 2:} We prove
$$H^1(v_{r-1}v_r, \mathbb Ch(\alpha_n)\oplus \mathbb C_{-\alpha_n})=H^1(v_{r-1}, V)$$ and 
$$H^1(v_{r-2}v_{r-1}v_r, \mathbb Ch(\alpha_n)\oplus \mathbb C_{-\alpha_n})=H^1(v_{r-2}v_{r-1}, V).$$
First note that $H^i(s_{j}\cdots s_{n-1}, \mathbb Ch(\alpha_n)\oplus \mathbb C_{-\alpha_n})=0$ for each $a_r\leq j\leq n-1$ for all $i\geq 1$
and $H^i(s_ns_{a_r}\cdots s_{n-1}, \mathbb Ch(\alpha_n)\oplus \mathbb C_{-\alpha_n})=0$ for all $i\geq 1$.
By using SES repeatedly, we see that $$H^1(v_{r-1}v_r, \mathbb Ch(\alpha_n)\oplus \mathbb C_{-\alpha_n})=H^1(v_{r-1}, H^0(v_r, \mathbb Ch(\alpha_n)\oplus \mathbb C_{-\alpha_n}))$$
and $$H^1(v_{r-2}v_{r-1}v_r, \mathbb Ch(\alpha_n)\oplus \mathbb C_{-\alpha_n})=H^1(v_{r-2}v_{r-1}, H^0(v_r, \mathbb Ch(\alpha_n)\oplus \mathbb C_{-\alpha_n})).$$
Hence we have $$H^1(v_{r-1}v_r, \mathbb Ch(\alpha_n)\oplus \mathbb C_{-\alpha_n})= H^1(v_{r-1}, V)$$ and
$$H^1(v_{r-2}v_{r-1}v_r, \mathbb Ch(\alpha_n)\oplus \mathbb C_{-\alpha_n})= H^1(v_{r-2}v_{r-1}, V).$$

From Step 1, Step 2 and Proposition \ref{H1alphan}(3), we see that  
$$H^1(v_{r-1}v_rs_n, \alpha_n)= H^1(v_{r-1}, V)\hspace{1cm}(6.4.1)$$ and 
$$H^1(w, \alpha_n)= H^0(v_{r-2}, H^1(v_{r-1}, V)).\hspace{1cm}(6.4.2)$$ 
 By Lemma \ref{H^1(2)},  $H^1(v_{r-1}, V)_{\mu}\neq 0$ if and if
  $\mu$ is of the form: \\ (i) $\mu=-(\sum\limits_{j=t}^{n}\alpha_j)$, where $a_{r-1}\leq t \leq n-1$
or \\
(ii) $\mu=-(\alpha_n+2(\sum\limits_{j=t_1}^{n-1}\alpha_j)+\sum\limits^{t_1-1}_{j=t_2}\alpha_j)$, where $a_{r-1}\leq t_2< t_1 \leq n-1$.

Proof of (1) is immediate from (6.4.1). 

Proof of (2):
If $\mu$ is of type (i), then $\langle \mu, \alpha_n \rangle=-1$. By Lemma \ref{lemma2.2}, (6.4.2) and SES, we see that $H^1(w, \alpha_n)_{\mu}=0.$

Fix $a_{r-2}\leq l \leq n-1$. By recursion, we assume that $H^0(s_{l+1}\ldots s_n, H^1(v_{r-1}, V))_{\mu}\neq 0$ if and only if 
$\mu=-(\alpha_n+2(\sum\limits_{j=t_1}^{n-1}\alpha_j)+\sum\limits^{t_1-1}_{j=t_2}\alpha_j)$ for some  $a_{r-1}\leq t_2< t_1 \leq l$.

On the other hand, we have $\langle -(\alpha_n+2(\sum\limits_{j=t_1}^{n-1}\alpha_j)+\sum\limits^{t_1-1}_{j=t_2}\alpha_j), \alpha_l \rangle =-1$.
Therefore, $H^0(s_l, H^0(s_{l+1}\cdots s_n, H^1(v_{r-1}, V)))_{\mu}\neq 0$ if and only if 
$\mu=-(\alpha_n+2(\sum\limits_{j=t_1}^{n-1}\alpha_j)+\sum\limits^{t_1-1}_{j=t_2}\alpha_j)$ for some $a_{r-1}\leq t_2< t_1 \leq l-1$.
%If $\mu$ is of type (ii), then $\langle \mu, \alpha_n \rangle=0$. By Lemma \ref{lemma2.2}, (6.4.2) and SES, we see that $H^1(w, \alpha_n)_{\mu}\neq 0.$
Now the proof of (2) follows by recursion.
%Hence we conclude the proof of the lemma.
\end{proof}

    \begin{proposition}\label{H^1}  
 Let $3\leq r \leq k$. Then,
$H^1(w_r, \alpha_n)_{\mu}\neq 0$ if and only if 
$\mu= (\prod\limits_{j=t_1}^n s_j)(\prod\limits_{j=t_2}^{n-1} s_j)(\alpha_{n-1})$
for some $a_{r-1}\leq t_2 < t_1\leq a_{r-2}-1.$
In such a case, $dim(H^1(w_r, \alpha_n)_{\mu})=1$.
  \end{proposition}
  \begin{proof}
  Let $u=(\prod\limits_{j=a_1}^n s_j)(\prod\limits_{j=a_2}^n s_j)(\prod\limits_{j=a_3}^{n} s_j)\cdots (\prod\limits_{j=a_{r-3}}^{n} s_j)$
  and let $v=(\prod\limits_{j=a_{r-2}}^n s_j)(\prod\limits_{j=a_{r-1}}^n s_j)(\prod\limits_{j=a_r}^{n} s_j)$.
 Observe that $w_r=uv$, $l(w_r)=l(uv)=l(u)+l(v)$.
 
 \underline {Claim:} $H^1(w_r, \alpha_n)=H^0(u, H^1(v, \alpha_n))$.
 
 By SES we have the following short exact sequence of $B$-modules:
 
 $0\longrightarrow H^1(s_n, H^0(v, \alpha_n))\longrightarrow H^1(s_nv, \alpha_n))\longrightarrow H^0(s_n, H^1(v, \alpha_n))\longrightarrow 0.$

 By Lemma \ref{H0alphan}, if $H^0(v, \alpha_n)_{\mu}\neq 0$, then we see that $\langle \mu, \alpha_j\rangle =0$ for all $a_{r-3} \leq j \leq n-1$.  
 Therefore, the vector bundle $\mathcal L(H^0(v, \alpha_n))$ on $X(u')$ is trivial for each $u'\leq u$.
 Thus, $H^i(u', H^0(v, \alpha_n))=0$ for each $u'\leq u$ and for all $i\geq 1$.

 Hence by using SES recursively, we see that $$H^1(uv, \alpha_n)=H^0(u, H^1(v, \alpha_n)).$$ 
 Therefore, we have $H^1(w_r, \alpha_n)=H^0(u, H^1(v, \alpha_n))$.
 
 By Lemma \ref{H1(3)}, $H^1(v, \alpha_n)_{\mu}\neq 0$ if and only if $\mu$ is of the form:
 
 $\mu=-(\alpha_n+2(\sum\limits_{j=t_1}^{n-1}\alpha_j)+\sum\limits^{t_1-1}_{j=t_2}\alpha_j)$, where $a_{r-1}\leq t_2< t_1 \leq a_{r-2}-1$.
 
 Note that $\langle \mu, \alpha_j \rangle=0$ for all $a_{r-2}\leq j \leq n$. Therefore, 
 $\mathcal L(H^1(v, \alpha_n))$ is the trivial vector bundle on $X(u)$. Hence $H^0(u, H^1(v, \alpha_n))=H^1(v, \alpha_n)$.
Thus, we conclude that $H^1(w_r, \alpha_n)_{\mu}\neq 0$ if and only if $\mu$ is of the form: \\ 
   $\mu=-(\alpha_n+2(\sum\limits_{j=t_1}^{n-1}\alpha_j)+\sum\limits^{t_1-1}_{j=t_2}\alpha_j)$, where $a_{r-1}\leq t_2< t_1 \leq a_{r-2}-1$.
 Note that $\mu=-(\alpha_n+2(\sum\limits_{j=t_1}^{n-1}\alpha_j)+\sum\limits^{t_1-1}_{j=t_2}\alpha_j)=(\prod\limits_{j=t_1}^n s_j)(\prod\limits_{j=t_2}^{n-1} s_j)(\alpha_{n-1})$.
 This completes the proof of the proposition.
   \end{proof}

\begin{corollary}\label{6.6}
Let $2\leq r\leq k$. If $H^1(w_r, \alpha_n)_{\mu}\neq 0$, then $H^0(\tau_r, \alpha_{n-1})_{\mu}\neq 0$.
\end{corollary}
\begin{proof} \underline{Case 1:} $3\leq r\leq k$.
 If $H^1(w_r, \alpha_n)_{\mu}\neq 0$ then by Proposition \ref{H^1}, we have 
 $\mu= (\prod\limits_{j=t_1}^n s_j)(\prod\limits_{j=t_2}^{n-1} s_j)(\alpha_{n-1})$
for some integers $t_2 < t_1$ such that  $a_{r-1}\leq t_2 < t_1\leq a_{r-2}-1.$  Then by Lemma \ref{Weights}, we have $H^0(\tau_r, \alpha_{n-1})_{\mu}\neq 0$.

\underline{Case 2:} $r=2$.
The proof follows from Lemma \ref{H0}(3) and Lemma \ref{H1(3)}(1).
\end{proof}

\begin{corollary}\label{H1toH0}
 Let $2\leq r\leq k$.
 If $H^1(w_r, \alpha_n)_{\mu}\neq 0$, then $H^0(w_{r-1}, \alpha_n)_{\mu} \neq 0$.
\end{corollary}
\begin{proof}
\underline{Case 1:} $3\leq r \leq k$.
If $H^1(w_r, \alpha_n)_{\mu}\neq 0$, by Proposition \ref{H^1}, $\mu= -(\alpha_n+2(\sum\limits_{j=t_1}^n \alpha_j)+
\sum\limits_{j=t_2}^{t_1-1}\alpha_j)$ for some integers $t_2 < t_1$ such that  $a_{r-1}\leq t_2<t_1 \leq a_{r-2}-1$.
For each such $\mu$, by Proposition \ref{H0alphan} (ii) 
, we see that $H^0(w_{r-1}, \alpha_n)_{\mu} \neq 0$.

\underline{Case 2:} $r=2$. First note that 
$$H^0(w_1, \alpha_n)=H^0(s_{a_1}\cdots s_{n-1}, \mathbb Ch(\alpha_n)\oplus \mathbb C_{-\alpha_n}) ~~\mbox{and}~~$$
$$H^1(w_2, \alpha_n)=H^1(s_{a_1}\cdots s_{n-1}, H^0(s_ns_{a_2}\cdots s_{n-1}, \mathbb Ch(\alpha_n)\oplus \mathbb C_{-\alpha_n}).$$ 

Now the proof follows from Lemma \ref{V_m}(1) and Lemma \ref{H^1(2)}.
\end{proof}

Recall $[a_k, n]=s_{a_k}\cdots s_n$ and $[a_k, n-1]=s_{a_k}\cdots s_{n-1}$.
Then, we have 

\begin{lemma}\label{w_k}
 $H^1(w_k[a_k, n], \alpha_n)=H^0(w_k[a_k, n-1], \alpha_{n-1})$.
\end{lemma}
\begin{proof}
 Note that $l(w_k[a_k, n]s_{n-1})=l(w_k[a_k, n])-1$.
 Since $\langle \alpha_n, \alpha_{n-1} \rangle =-2$, by Lemma \ref{lemma2.2}, we have 
 $$H^1(w_k[a_k, n], \alpha_n)=H^0(w_k[a_k, n], s_{\alpha_{n-1}}\cdot \alpha_n) .\hspace{0.8cm}(6.8.1)$$
 By SES and $s_{\alpha_{n-1}}\cdot \alpha_n=\alpha_n+\alpha_{n-1}$, we have 
  $$H^0(w_k[a_k, n], s_{\alpha_{n-1}}\cdot \alpha_n)=H^0(w_k[a_k, n-1], H^0(s_{n}, \alpha_{n-1}+\alpha_n)).\hspace{0.8cm}(6.8.2)$$
By applying  $H^0(w_k[a_k, n-1],-)$ to the following short exact sequence, 
$$0\longrightarrow \mathbb C_{\alpha_{n-1}}\longrightarrow H^0(s_n, \alpha_n+\alpha_{n-1})\longrightarrow \mathbb C_{\alpha_{n-1}+\alpha_n}\longrightarrow 0.$$
 We obtain the following long exact sequence:
 
$0\longrightarrow H^0(w_k[a_k, n-1], \alpha_{n-1})\longrightarrow H^0(w_k[a_k, n-1], H^0(s_n, \alpha_n+\alpha_{n-1}))
\longrightarrow H^0(w_k[a_k, n-1], \alpha_{n-1}+\alpha_n)\longrightarrow H^1(w_k[a_k, n-1], \alpha_{n-1})\longrightarrow \cdots .$

Note that $\langle \alpha_n+\alpha_{n-1}, \alpha_{n-1}\rangle =0$ and 
$\langle \alpha_n+\alpha_{n-1}, \alpha_{n-2}\rangle =-1$. Hence by Lemma \ref{lemma2.2} and SES, we see that 
$$H^0(w_k[a_k, n-1], \alpha_{n-1}+\alpha_n)=0.$$
Therefore, we have $$H^0(w_k[a_k, n-1], \alpha_{n-1})=H^0(w_k[a_k, n-1], H^0(s_n, \alpha_n+\alpha_{n-1})).\hspace{0.8cm}(6.8.3)$$
Hence by using (6.8.1), (6.8.2) and (6.8.3), we  conclude that $$H^1(w_k[a_k, n], \alpha_n)=H^0(w_k[a_k, n-1], \alpha_{n-1}).\qedhere$$
\end{proof}

Let $1\leq r \leq k$. Let $M_r:=\{\mu\in X(T): H^1(w_r, \alpha_n)_{\mu}\neq 0\}$ and $M_0:=\{\mu\in X(T): H^1(w_k[a_k, n], \alpha_n)_{\mu}\neq 0\}$.
Then, we have 

\begin{corollary}\label{6.10} \

\begin{enumerate}
 \item 
  $M_r\cap M_{r'} =\emptyset$ whenever  $r\neq r'$.
\item $M_0\cap M_r = \emptyset$ for every $1\leq r \leq k$.
\end{enumerate}
 \end{corollary}
\begin{proof} Note that by Proposition \ref{H1alphan}(1), $H^1(w_1, \alpha_n)=0$.
  By Lemma \ref{w_k}, we have $$H^1(w_k[a_k, n], \alpha_n)=H^0(w_k[a_k, n-1], \alpha_{n-1}).$$
 Now the proof follows from  Lemma \ref{Weights}, Lemma \ref{H0u_1} and   Corollary \ref{6.6}.
\end{proof}

\begin{lemma}\label{lemma4.9} Let $u, w\in W$, let $v:=(\prod\limits_{j=1}^{n}s_j)^{l}$  for some positive integer $l\leq n$ such that $w=uv$ and $l(w)=l(u)+l(v)$. 
If $l\geq 3$, then  $H^i(w, \alpha_{n})=0$ for all $i\geq 0$.
\end{lemma}
\begin{proof}
First note that by SES, $$H^0(w, \alpha_n)=H^0(u, H^0(v, \alpha_n)).$$ By similar arguments as in the proof of Lemma \ref{H1(3)}, we see that $$H^1(w, \alpha_n)=H^0(u, H^1(v, \alpha_n)).$$
We now prove that $H^0(v, \alpha_n)=0$ and $H^1(v, \alpha_n)=0$.
Let $c=\prod\limits_{j=1}^ns_j$.
Note that for each $1\leq r \leq n$,  $c^r(\alpha_j)<0$ for $n+1-r\leq j \leq n$. 
In particular, we have $l(vs_{n-1})=l(v)-1$ and $l(vs_{n-1}s_{n-2})=l(v)-2$.

Therefore, by Lemma \ref{lemma2.2} and using SES, we have $$H^0(v, \alpha_n)=H^0(vs_{n-1}, H^0(s_{n-1}, \alpha_n))=0.$$
Further, by Lemma \ref{lemma2.2}, we have $$H^1(v, \alpha_n)=H^0(vs_{n-1}, H^1(s_{n-1}, \alpha_n))=H^0(vs_{n-1}, \mathbb C_{\alpha_{n-1}+\alpha_n}).$$
Since $\langle \alpha_{n-1}+\alpha_n, \alpha_{n-2} \rangle =-1$ and $l(vs_{n-1}s_{n-2})=l(vs_{n-1})-1; $ by Lemma \ref{lemma2.2}, we have 
$H^0(vs_{n-1}, \alpha_{n-1}+\alpha_n)=0$.
Hence $H^1(v, \alpha_n)=0$.

Therefore, by \cite[Corollary 6.4]{Ka4},  we see that $H^i(v, \alpha_n)=0$ for all $i\geq 0$.
Hence, we conclude that $H^i(w, \alpha_n)=0$ for all $i\geq 0$.
 \end{proof}

\section{Cohomology modules of the tangent bundle of $Z(w, \underline i)$}

Let $w\in W$ and let $w =s_{i_1}s_{i_2}\cdots s_{i_r}$ 
be a reduced expression for $w$ and let $\underline i=(i_1, i_2, \ldots, i_r)$.
Let $\tau=s_{i_1}s_{i_2}\cdots s_{i_{r-1}}$ and
$\underline i'=(i_1, i_2, \ldots, i_{r-1})$. 

Recall the following long exact sequence of $B$-modules from \cite{CKP} (see \cite[Proposition 3.1]{CKP}):

 $0\longrightarrow H^0(w, \alpha_{i_r})\longrightarrow H^0(Z(w, \underline i), T_{(w,\underline i)})\longrightarrow H^0(Z(\tau, \underline i'), T_{(\tau, \underline i')})\longrightarrow H^1(w, \alpha_{i_r})\longrightarrow 
 H^1(Z(w, \underline i), T_{(w,\underline i)})\longrightarrow H^1(Z(\tau, \underline i), T_{(\tau,\underline i')})
 \longrightarrow H^2(w, \alpha_{i_r})\longrightarrow H^2(Z(w, \underline i), T_{(w,\underline i)}) \longrightarrow  H^2(Z(\tau, \underline i'), T_{(\tau, \underline i')})\longrightarrow H^3(w, \alpha_{i_r})\longrightarrow \cdots$.

 By \cite[Corollary  6.4]{Ka4}, we have $H^j(w, \alpha_{i_r})=0$ for every $j\geq 2$. 
 Thus we have the following exact sequence of $B$-modules:
  
   $0\longrightarrow H^0(w, \alpha_{i_r})\longrightarrow H^0(Z(w, \underline i), T_{(w,\underline i)})\longrightarrow H^0(Z(\tau, \underline i'), T_{(\tau, \underline i')})\longrightarrow H^1(w, \alpha_{i_r})\longrightarrow 
 H^1(Z(w, \underline i), T_{(w,\underline i)})\longrightarrow H^1(Z(\tau, \underline i), T_{(\tau,\underline i')})
 \longrightarrow 0.$

 Now onwards we call this exact sequence by LES.

Let $w_0=s_{j_1}s_{j_2}\cdots s_{j_N}$ be a reduced expression of $w_0$ and $\underline j=(j_1, j_2, \ldots, j_N)$ such that $(j_1, j_2, \ldots, j_r)=\underline i$.
 
 \begin{lemma}\label{T}
  The natural homomorphism $$f: H^1(Z(w_0, \underline j), T_{(w_0, \underline j)})\longrightarrow H^1(Z(w, \underline i), T_{(w, \underline i)})$$ of $B$-modules is surjecive.
 \end{lemma}
\begin{proof}
If $w=w_{0}$, we are done. Otherwise, let $\underline {i}_1=(j_1, \ldots, j_r, j_{r+1})$.
Note that by \cite[Corollary 6.4]{Ka4}, we have $H^2(ws_{i_{r+1}}, \alpha_{r+1})=0$. 
Then by LES, the natural  homomorphism 
$$H^1(Z(ws_{i_{r+1}}, \underline i_1), T_{(ws_{i_{r+1}},\underline i_1)})\longrightarrow H^1(Z(w, \underline i), T_{(w,\underline i)})$$
is surjective. 
 By descending induction on $l(w)$, the natural homomorphism $H^1(Z(w_0, \underline j), T_{(w_0, \underline j)})\longrightarrow 
 H^1(Z(ws_{i_{r+1}}, \underline i_1), T_{(ws_{i_{r+1}}, \underline i_1)})$  of $B$-modules is surjecive.
 Hence the  natural  homomorphism 
 $$f: H^1(Z(w_0, \underline j), T_{(w_0, \underline j)})\longrightarrow H^1(Z(w, \underline i), T_{(w, \underline i)})$$ of $B$-modules is surjective.
\end{proof}

\begin{lemma}\label{lemma7.2}
 Let $J=S\setminus \{\alpha_n\}$.  Let $v\in W_J$ and $u\in W$ be such that $l(uv)=l(u)+l(v)$.
 Let $u=s_{i_1}\cdots s_{i_r}$ and $v=s_{i_{r+1}}\cdots s_{i_t}$ be reduced expressions of $u$ and $v$ respectively.
 Let $\underline i=(i_1, i_2, \ldots, i_{r})$ and $\underline j=(i_1, i_2, \ldots, i_{r}, i_{r+1}, \ldots, i_t)$. Then, we have 
 \begin{enumerate}
  \item 
 The natural homomorphism
 $$H^0(Z(uv, \underline j), T_{(uv, \underline j)})\to H^0(Z(u, \underline i), T_{(u, \underline i)})$$
of $B$-modules is surjecive.
\item 
The natural homomorphism $$H^1(Z(uv, \underline j), T_{(uv, \underline j)})\longrightarrow H^1(Z(u, \underline i), T_{(u, \underline i)})$$
of $B$-modules is an isomorphism.
 \end{enumerate}
 \end{lemma}
\begin{proof}
Let $r+1\leq l \leq t$. Let $v_l=us_{i_{r+1}}\cdots s_{i_l}$
 and $\underline{i}_{l}=(\underline{i}, i_{r+1}, \ldots, i_{l})$.
 
Proof of (1):
 By Corollary \ref{short},  we have $H^{i}(v_t, \alpha_{i_{t}})=0$.  Therefore, using LES, 
  we see that the natural homomorphism $$H^0(Z(v_t, \underline i_{t}), T_{(v_{t}, \underline i_{t})})\to H^0(Z(v_{t-1}, \underline i_{t-1}), T_{(v_{t-1}, \underline i_{t-1})})$$ is surjective. By the recursion, the natural homomorphism  $$H^0(Z(v_{t-1}, \underline i_{t-1}), T_{(v_{t-1}, \underline i_{t-1})})\to H^0(Z(u, \underline i), T_{(u, \underline i)})$$ is surjective. Hence we conclude that the natural homomorphism 
 $$H^0(Z(uv, \underline j), T_{(uv, \underline j)})\to H^0(Z(u, \underline i), T_{(u, \underline i)})$$ is  surjective.

 Proof of (2): 
 Proof is by induction on $l(v)$.
  By LES, we have the following exact sequence of $B$-modules:
 
   $0\longrightarrow H^0(uv, \alpha_{i_t})\longrightarrow H^0(Z(uv, \underline j), T_{(uv,\underline j)})\longrightarrow 
   H^0(Z(v_{t-1}, \underline i_{t-1}), T_{(v_{t-1}, \underline i_{t-1})})\longrightarrow H^1(uv, \alpha_{i_t})\longrightarrow 
 H^1(Z(uv, \underline j), T_{(uv,\underline j)})\longrightarrow H^1(Z(v_{t-1}, \underline i_{t-1}), T_{(v_{t-1},\underline i_{t-1})})
 \longrightarrow 0.$

  By induction on $l(v)$, the natural homomorphism 
  $H^1(Z(v_{t-1}, \underline i_{t-1}), T_{(v_{t-1}, \underline i_{t-1})})\longrightarrow H^1(Z(u, \underline i), T_{(u, \underline i)})$ is an isomorphism.
 
 By Corollary \ref{short}, $H^1(uv, \alpha_{i_t})=0$. Therefore, by the above exact sequence, we see that 
 $H^1(Z(uv, \underline j), T_{(uv,\underline j)})\longrightarrow H^1(Z(v_{t-1}, \underline i_{t-1}), T_{(v_{t-1},\underline i_{t-1})})$ is an isomorphism.
Hence, we conclude that the homomorphism
$$H^1(Z(uv, \underline j), T_{(uv, \underline j)})\longrightarrow H^1(Z(u, \underline i), T_{(u, \underline i)})$$
of $B$-modules is an isomorphism.
 \end{proof}

Recall that by Lemma \ref{lemma: 1} and  Lemma \ref{coxeter} we have
$$w_0=(\prod_{l_1=1}^{k-1}[a_{l_1}, n])([a_k, n]^{n+1-k})(\prod_{l_2=1}^{k-1}[a_{k}, a_{l_2}-1])$$ 
is a reduced expression for $w_0$. Let $\underline i$ be the tuple corresponding to this reduced expression of $w_0$.
Let $u_1=w_k[a_k, n]$ and $\underline i_1$ be the tuple corresponding to the reduced expression 
$(\prod_{l_1=1}^{k}[a_{l_1}, n])([a_k, n])$.
Note that $a_k=1$. With this notation, we have

\begin{lemma}\label{7.3} \

\begin{enumerate}
 \item 
 The natural homomorphism $$H^0(Z(w_0, \underline{i}),T_{(w_0, \underline i)})\longrightarrow H^0(Z(u_1, \underline{i}_1),T_{(u_1, \underline i_1)})$$ of $B$-modules is an isomorphism. 
 \item The natural homomorphism $$H^1(Z(w_0, \underline{i}),T_{(w_0, \underline i)})\longrightarrow H^1(Z(u_1, \underline{i}_1),T_{(u_1, \underline i_1)})$$ of $B$-modules is an isomorphism. 
 
\end{enumerate}
 \end{lemma}
\begin{proof}
 Let $u_j=w_k[a_k, n]^j$ and $\underline i_j$ be the tuple corresponding to the reduced expression $(\prod_{l_1=1}^{k}[a_{l_1}, n])([a_k, n]^j)$ (see Lemma \ref{coxeter}).
 By Lemma \ref{lemma7.2}, the natural homomorphism
 $$H^0(Z(w_0, \underline{i}),T_{(w_0, \underline i)})\longrightarrow H^0(Z(u_{n-k}, \underline{i}_{n-k}),T_{(u_{n-k}, \underline i_{n-k})})$$ is  surjective
 and the natural homomorphism $$H^1(Z(w_0, \underline{i}),T_{(w_0, \underline i)})\longrightarrow H^1(Z(u_{n-k}, \underline{i}_{n-k}),T_{(u_{n-k}, \underline i_{n-k})})$$
is an isomorphism.

 If $j\geq 2$, then by Lemma \ref{lemma4.9}, we have $H^1(u_j, \alpha_n)=0$. 
 Hence by LES and Lemma \ref{lemma7.2}, for each $2\leq j\leq n-k$
 we observe that the natural homomorphism 
 $$H^0(Z(u_j, \underline{i}_{j}),T_{(u_{j}, \underline i_{j})})\longrightarrow H^0(Z(u_{j-1}, \underline{i}_{j-1}),T_{(u_{j-1}, \underline i_{j-1})})$$
is surjective and $$H^1(Z(u_j, \underline{i}_{j}),T_{(u_{j}, \underline i_{j})})\longrightarrow H^1(Z(u_{j-1}, \underline{i}_{j-1}),T_{(u_{j-1}, \underline i_{j-1})})$$
is an isomorphism.
Therefore, the homomorphism $$H^0(Z(w_0, \underline{i}),T_{(w_0, \underline i)})\longrightarrow H^0(Z(u_1, \underline{i}_1),T_{(u_1, \underline i_1)})$$ of $B$-modules is
surjecive and the homomorphism 
$$H^1(Z(w_0, \underline{i}),T_{(w_0, \underline i)})\longrightarrow H^1(Z(u_1, \underline{i}_1),T_{(u_1, \underline i_1)})$$ of $B$-modules is an isomorphism. 
  
  Further since $u_1^{-1}(\alpha_0)<0$, by \cite[Lemma 6.2]{CKP}, we have $H^0(Z(u_1, \underline{i}_1),T_{(u_1, \underline i_1)})_{-\alpha_0}\neq 0$.
  By \cite[Theorem 7.1]{CKP}, $H^0(Z(w_0, \underline{i}),T_{(w_0, \underline i)})$ is a parabolic Lie subalgebra of $\mathfrak{g}$
  and hence there is a unique $B$-stable line in $H^0(Z(w_0, \underline{i}),T_{(w_0, \underline i)})$, namely $\mathfrak{g}_{-\alpha_0}$.
  Therefore we conclude that the natural homomorphism 
  $$H^0(Z(w_0, \underline{i}),T_{(w_0, \underline i)})\longrightarrow H^0(Z(u_1, \underline{i}_1),T_{(u_1, \underline i_1)})$$ of $B$-modules is
  an isomorphism.  
  \end{proof}

The following is a useful corollary.

\begin{corollary} \label{A}
   If $\mu\in X(T)\setminus \{0\}$, then $dim(H^0(Z(u_1, \underline{i}_1),T_{(u_1, \underline i_1)})_{\mu})\leq 1$. 
\end{corollary}
\begin{proof}
 By \cite[Theorem 7.1]{CKP}, $H^0(Z(w_0, \underline{i}),T_{(w_0, \underline i)})$ is a parabolic Lie subalgebra of $\mathfrak{g}$.
 By Lemma \ref{7.3}(1), we have 
 $$H^0(Z(w_0, \underline{i}),T_{(w_0, \underline i)})\simeq H^0(Z(u_1, \underline{i}_1),T_{(u_1, \underline i_1)}) ~~\mbox{(as $B$-modules)}. $$
  Hence for any $\mu\in X(T)\setminus\{0\}$, 
 we have $$dim(H^0(Z(u_1, \underline{i}_1),T_{(u_1, \underline i_1)})_{\mu})\leq 1.\qedhere$$
 \end{proof}

Let  $u'_1=w_k[a_k, n-1]$ and let $\underline{i'}_1$ be the tuple corresponding to the reduced expression $(\prod_{l_1=1}^{k}[a_{l_1}, n])[a_k, n-1]$. 
Then, we have  

\begin{lemma}\label{7.4}
Let  $\mu\in X(T)\setminus \{0\}$.
\begin{enumerate} 
 \item 
If $H^1(u_1, \alpha_n)_{\mu}=0$, then 
$dim(H^0(Z(u'_1,\underline{i'}_1), T_{(u'_1, \underline{i'}_1)})_{\mu})\leq 1$.
\item If $H^1(u_1, \alpha_n)_{\mu}\neq 0$, then $dim(H^0(Z(u'_1,\underline{i'}_1), T_{(u'_1, \underline{i'}_1)})_{\mu})=2$.
\end{enumerate}
 \end{lemma}
\begin{proof}
 By LES, we have the following long exact sequence of $B$-modules:

 $0\longrightarrow H^0(u_1, \alpha_n)\longrightarrow H^0(Z(u_1, \underline{i}_1),T_{(u_1, \underline i_1)})\longrightarrow 
 H^0(Z(u'_1,\underline{i'}_1), T_{(u'_1, \underline{i'}_1)})\longrightarrow H^1(u_1, \alpha_n)\longrightarrow \cdots . \hspace{1cm}(7.5.1)$
 
 Proof of (1): If $H^1(u_1, \alpha_n)_{\mu}=0$ and $\mu\in X(T)\setminus \{0\}$, then by the above exact sequence
 the natural 
 homomorphism (of $T$-modules)
 $H^0(Z(u_1, \underline{i}_1),T_{(u_1, \underline i_1)})_{\mu}\longrightarrow H^0(Z(u'_1,\underline{i'}_1), T_{(u'_1, \underline{i'}_1)})_{\mu}$
is surjective.
 By Corollary \ref{A}, we have $dim(H^0(Z(u_1, \underline{i}_1),T_{(u_1, \underline i_1)})_{\mu})\leq 1$.
 Hence we conclude that $dim(H^0(Z(u'_1,\underline{i'}_1), T_{(u'_1, \underline{i'}_1)})_{\mu})\leq 1$.
 
 Proof of (2): If $H^1(u_1, \alpha_n)_{\mu}\neq 0$, then by Proposition \ref{H1alphan},  we have  $\mu\in R^-_s\setminus (-S)$. 
 Also note that  $dim(H^1(u_1, \alpha_n)_{\mu})=1$.  Hence by the above arguments, we see that if $H^1(u_1, \alpha_n)_{\mu}\neq 0$, 
 then $$dim(H^0(Z(u'_1,\underline{i'}_1), T_{(u'_1, \underline{i'}_1)})_{\mu})\leq 2. \hspace{1cm} (7.5.2)$$ Therefore, we have the following observations;
 
 \begin{itemize}
  \item Note that $H^0(u_1, \alpha_n)=0$.
  Hence by (7.5.1), we see that $H^0(Z(u_1, \underline{i}_1),T_{(u_1, \underline i_1)})$ is a $B$-submodule of $H^0(Z(u'_1,\underline{i'}_1), T_{(u'_1, \underline{i'}_1)}).$
\item $H^0(u_1', \alpha_{n-1})$ is a $B$-submodule of $H^0(Z(u'_1,\underline{i'}_1), T_{(u'_1, \underline{i'}_1)})$.
 \item Since $-\alpha_0$ is a long root, we note that $H^0(u_1', \alpha_{n-1})_{-\alpha_0}=0$.
 \end{itemize}

 By Lemma \ref{7.3},   we have $H^0(Z(w_0, \underline{i}),T_{(w_0, \underline i)})=H^0(Z(u_1, \underline{i}_1),T_{(u_1, \underline i_1)})$. 
 By \cite[Theorem 7.1]{CKP}, $H^0(Z(w_0, \underline{i}),T_{(w_0, \underline i)})$ is a parabolic Lie subalgebra of $\mathfrak{g}$  and hence it has a
 unique $B$-stable one dimensional subspace, namely $\mathfrak{g}_{-\alpha_0}$. 
  By above discussion, the $B$-submodule $H^0(u_1', \alpha_{n-1})\cap H^0(Z(u_1, \underline{i}_1),T_{(u_1, \underline i_1)})$ of $H^0(Z(w_0, \underline{i}),T_{(w_0, \underline i)})$
 does not contain $\mathfrak{g}_{-\alpha_0}$ and so it is zero.
 
 By Lemma \ref{w_k}, we have $H^1(u_1, \alpha_n)=H^0(u_1', \alpha_{n-1})$.  Therefore,  if $H^1(u_1, \alpha_n)_{\mu}\neq 0$, then we have
 $$dim(H^0(Z(u'_1,\underline{i'}_1), T_{(u'_1, \underline{i'}_1)})_{\mu})\geq 2 .\hspace{1cm} (7.5.3)$$

 Hence from (7.5.2) and (7.5.3), we conclude that $$dim(H^0(Z(u'_1,\underline{i'}_1), T_{(u'_1, \underline{i'}_1)})_{\mu})= 2 .$$ 
 This completes the proof of (2).
 \end{proof}
 \begin{corollary}\label{B}
  The natural homomorphism
  $$H^0(Z(u'_1,\underline{i'}_1), T_{(u'_1, \underline{i'}_1)})\longrightarrow H^1(u_1, \alpha_n)$$ is surjective.
 \end{corollary}
 \begin{proof}
 Note that by Proposition \ref{H1alphan}, if $H^1(u_1, \alpha_n)_{\mu}\neq 0$ then $\mu\in R^-_s\setminus (-S)$. 
By Corollary \ref{A}, we have $dim(H^0(Z(u_1, \underline{i}_1),T_{(u_1, \underline i_1)})_{\mu})\leq 1$.
By Lemma \ref{7.4}, if $H^1(u_1, \alpha_n)_{\mu}\neq 0$, then $$dim(H^0(Z(u'_1,\underline{i'}_1), T_{(u'_1, \underline{i'}_1)})_{\mu})= 2.$$
Hence by the exact sequence (7.5.1) we conclude that the natural homomorphism 
  $$H^0(Z(u'_1,\underline{i'}_1), T_{(u'_1, \underline{i'}_1)})_{\mu}\longrightarrow H^1(u_1, \alpha_n)_{\mu}$$ is surjective.
  \end{proof}

Let $1\leq r\leq k$.
Let $\underline j_{r}=(i_{a_1}, \ldots, n, i_{a_2}, \ldots, n,\ldots, i_{a_r}, \ldots, n)$ and   $\underline j'_{r}=(i_{a_1}, \ldots, n, i_{a_2}, \ldots, n,\ldots, i_{a_{r-1}}, \ldots, n, i_{a_r}, \ldots, n-1)$.

We now prove 
\begin{lemma} \label{7.5}\

\begin{enumerate} 
 \item If $H^1(w_m, \alpha_n)_{\mu}=0$ for all $r\leq m\leq k$ and $H^1(u_1, \alpha_n)_{\mu}= 0$, then 
 $dim(H^0(Z(\tau_{r}, \underline j_{r}'), T_{(\tau_{r}, \underline j_{r}')})_{\mu})\leq 1$.
 \item
 If $H^1(w_r, \alpha_n)_{\mu}\neq 0$, then 
 $dim(H^0(Z(\tau_{r}, \underline j_r'), T_{(\tau_{r}, \underline j_r')})_{\mu})=2$ and hence the natural homomorphism $$H^0(Z(\tau_{r}, \underline j_r'), T_{(\tau_{r}, \underline j_r')})_{\mu}\longrightarrow H^1(w_r, \alpha_n)_{\mu}$$ is surjective. 

\end{enumerate}

\end{lemma}
\begin{proof}
Proof of (1):
If $H^1(u_1, \alpha_n)_{\mu}= 0$. 
Then by Lemma \ref{7.4}, we have $dim(H^0(Z(u'_1,\underline{i'}_1), T_{(u'_1, \underline{i'}_1)})_{\mu})\leq 1$.
By Lemma \ref{lemma7.2},  the natural homomorphism 
$$H^0(Z(u'_1,\underline{i'}_1), T_{(u'_1, \underline{i'}_1)})_{\mu}\longrightarrow H^0(Z(w_k,\underline{j}_k), T_{(w_k, \underline{j}_k)})_{\mu}$$
is surjective. 
 If  $H^1(w_m, \alpha_n)_{\mu}=0$ for all $r\leq m\leq k$, by using LES, we see that the natural homomorphism 
$$H^0(Z(w_k,\underline{j}_k), T_{(w_k, \underline{j}_k)})_{\mu}\longrightarrow H^0(Z(\tau_{r}, \underline j_{r}'), 
T_{(\tau_{r}, \underline j_{r}')})_{\mu}$$ is surjective.
Therefore, we have  $dim(H^0(Z(\tau_{r}, \underline j_{r}'), T_{(\tau_{r}, \underline j_{r}')})_{\mu})\leq 1$.

 Proof of (2):
 If $H^1(w_r, \alpha_n)_{\mu}\neq 0$, then by Corollary \ref{6.10}, we have $H^1(w_m, \alpha_n)_{\mu}=0$ for all $r+1\leq m\leq k$ and $H^1(u_1, \alpha_n)_{\mu}= 0$. 
Then by (1) , we have $$dim(H^0(Z(\tau_{r+1}, \underline j_{r+1}'), T_{(\tau_{r+1}, \underline j_{r+1}')})_{\mu})\leq 1.$$
 
  By Lemma \ref{lemma7.2},  the natural homomorphism 
$$H^0(Z(\tau_{r+1}, \underline j_{r+1}'), T_{(\tau_{r+1}, \underline j_{r+1}')})_{\mu}\longrightarrow H^0(Z(w_r,\underline{j}_r), 
T_{(w_r, \underline{j}_r)})_{\mu}$$ is surjecive.
Hence $dim(H^0(Z(w_r,\underline{j}_r), T_{(w_r, \underline{j}_r)})_{\mu})\leq 1.$

By  LES we have the following long exact sequence of $B$-modules:

$0\longrightarrow H^0(w_r, \alpha_n)\longrightarrow H^0(Z(w_r,\underline{j}_r), T_{(w_r, \underline{j}_r)})\longrightarrow
H^0(Z(\tau_{r}, \underline j_r'), T_{(\tau_{r}, \underline j_r')})\longrightarrow H^1(w_r, \alpha_n)\longrightarrow \cdots .$

Since $dim(H^0(Z(w_r,\underline{j}_r), T_{(w_r, \underline{j}_r)})_{\mu})\leq 1$ and $dim(H^1(w_r, \alpha_n)_{\mu})=1$, we see that 
$dim(H^0(Z(\tau_{r}, \underline j_r'), T_{(\tau_{r}, \underline j_r')})_{\mu})\leq 2\hspace{1cm}(7.7.1)$.

Let $\underline j_r''$ be the tuple corresponding to the reduced expression $(\prod_{l_1=1}^{k-1}[a_{l_1}, n])[a_k, n-2]$ of $\tau_{r}s_{n-1}$.
 By LES we have the following exact sequence of $B$-modules:

$0\longrightarrow H^0(\tau_r, \alpha_{n-1})\longrightarrow H^0(Z(\tau_{r}, \underline j_r'), T_{(\tau_{r}, \underline j_r')}) \longrightarrow
H^0(Z(\tau_{r}s_{n-1}, \underline j_r''), T_{(\tau_{r}s_{n-1}, \underline j_r'')})\longrightarrow 0.\hspace{1cm}(7.7.2)$
 
 By Lemma \ref{lemma7.2}, the natural homomorphism
 $H^0(Z(\tau_{r}s_{n-1}, \underline j_r''), T_{(\tau_{r}s_{n-1}, \underline j_r'')}) \longrightarrow H^0(Z(w_{r-1}, \underline j_{r-1}), T_{(w_{r-1}, \underline j_{r-1})})$ is surjective.

 Observe that, by LES we have 
 $H^0(w_{r-1}, \alpha_n)$ is a $B$-submodule of $H^0(Z(w_{r-1}, \underline j_{r-1}), T_{(w_{r-1}, \underline j_{r-1})})$.
 If $H^1(w_r, \alpha_n)_{\mu}\neq 0$ then Corollary \ref{6.6}, we have
 $H^0(\tau_r, \alpha_{n-1})_{\mu}\neq 0$ and by Corollary \ref{H1toH0}, we have  $H^0(w_{r-1}, \alpha_n)_{\mu}\neq 0$. Therefore, we have
 $dim(H^0(Z(\tau_{r}s_{n-1}, \underline j_r''), T_{(\tau_{r}s_{n-1}, \underline j_r'')})_{\mu})\geq dim(H^0(Z(w_{r-1}, \underline j_{r-1}), T_{(w_{r-1}, \underline j_{r-1})})_{\mu})\geq 1.$
    Hence by above exact sequence (7.7.1), we observe that 
$$dim(H^0(Z(\tau_{r}, \underline j_r'), T_{(\tau_{r}, \underline j_r')})_{\mu})\geq 2\hspace{1cm}(7.7.3).$$
 
 Therefore, from $(7.7.1)$ and $(7.7.3)$ we conclude that  $dim(H^0(Z(\tau_{r}, \underline j_r'), T_{(\tau_{r}, \underline j_r')})_{\mu})=2$.
This completes the proof of (2).
\end{proof}

\section{Main Theorem}
In this section we prove the main theorem.

\noindent Recall that $G=PSp(2n, \mathbb C)$ $(n\geq 3)$, and let  $c$ be a Coxeter element in $W$. Also, recall that $c$ has a reduced expression  $c=[a_1, n][a_2, a_{1}-1]\cdots [a_k, a_{k-1}-1]$, 
where $[i, j]$ for $i\leq j$ denotes $s_is_{i+1}\cdots s_j$
and $n\geq a_1>a_2>\ldots >a_k=1$.

\noindent Let $\underline i=(\underline i^1, \underline i^2,\ldots, \underline i^n)$ be a sequence corresponding 
to a reduced expression of $w_0$, where
$\underline i^r$ $(1\leq r\leq n)$ is a sequence of reduced expressions of $c$ (see Lemma \ref{coxeter}). 
  Then, we have 
  
\begin{theorem}\label{Theorem} 
  
$H^j(Z(w_0, \underline i), T_{(w_0, \underline i)})=0$ for all $j\geq 1$ if and only if 
$a_{1}\neq n-1$ and $a_{2}\leq n-2$.
\end{theorem}
\begin{proof}
By \cite[Proposition 3.1]{CKP}, we have $H^j(Z(w_0, \underline i), T_{(w_0, \underline i)})=0$ for all $j\geq 2$. So, by Corollary \ref{com1}(2), it is enough to prove the following:

$H^1(Z(w_0, \underline i), T_{(w_0, \underline i)})=0$ for all $j\geq 1$ if and only if $c$ is 
of the form $[a_1, n][a_2, a_{1}-1]\cdots [a_k, a_{k-1}-1]$ with 
$a_{1}\neq n-1$ and $a_{2}\leq n-2$.

Proof of ($\Longrightarrow$): If $a_1=n-1$, then $c=s_{n-1}s_{n}v$ for some $v\in W_{J}, $ where 
$J=S\setminus \{\alpha_n, \alpha_{n-1}\}$. Take $w=s_{n-1}s_n$ and $\underline i'=(n-1, n)$ in LES, then we have
    
 $0\longrightarrow H^0(w, \alpha_n) \longrightarrow H^0(Z(w, \underline i'), 
 T_{(w,\underline i')})\longrightarrow H^0(s_{n-1}, \alpha_{n-1})\longrightarrow
 H^1(w, \alpha_{n})\overset {f}\longrightarrow 
 H^1(Z(w, \underline i'), T_{(w,\underline i')})\longrightarrow  0$.
 
By a simple calculation, we see that $H^1(s_{n-1}s_{n}, \alpha_{n})= \mathbb C_{\alpha_n+\alpha_{n-1}}$ and 
$H^0(s_{n-1}, \alpha_{n-1})_{\alpha_n+\alpha_{n-1}}=0$.
Hence $f$ is a non zero homomorphism. 
Hence $H^1(Z(w, \underline i'), T_{(w,\underline i')})\neq 0$.
By Lemma \ref{T}, the natural homomorphism $$H^1(Z(w_0, \underline i), T_{(w_0,\underline i)})\longrightarrow 
H^1(Z(w, \underline i'), T_{(w,\underline i')})$$ is surjective.
Hence we have
 $$H^1(Z(w_0, \underline i), T_{(w_0,\underline i)})\neq 0.$$
 If $a_2=n-1$, then  $c=s_{n}s_{n-1}v$ for some $v\in W_{J}, $ where $J=S\setminus \{\alpha_n, \alpha_{n-1}\}$.
Let $\underline i''$ be a reduced expression of $c=s_{n}s_{n-1}v$ and $\underline i'''$ be a  
be the reduced expression of $cs_n$ such that $\underline i'''=(i_1', \ldots, i'_n, n)$ with 
$(i_1', \ldots, i'_n)=\underline i''.
$
Since $\langle \alpha_{n}, \alpha_{j} \rangle=0$ for $j\in \{1,2,\ldots, n-2\}$, we have 
  $$H^i(cs_{n}, \alpha_n)=H^i(s_ns_{n-1}s_{n}, \alpha_n) ~ \mbox{for} ~  \geq 0.$$
  By {\it SES}, we see that $H^1(s_ns_{n-1}s_{n}, \alpha_n)=\mathbb C_{\alpha_{n-1}}$.
 By \cite[Proposition 6.3]{CKP}, we have 
 $$H^0(Z(c, \underline i''), T_{(c, \underline i'')})_{\alpha_{n-1}}=0.$$
Hence, by LES, we conclude that $H^1(Z(cs_{n}, \underline i'''), T_{(cs_{n}, \underline i''')})_{\alpha_{n-1}}\neq 0$. 
By Lemma \ref{T}, the natural homomorphism $H^1(Z(w_0, \underline i), T_{(w_0,\underline i)})\longrightarrow 
H^1(Z(cs_{n}, \underline i'''), T_{(cs_{n},\underline i''')})$ is surjective. Hence, we have 
 $$H^1(Z(w_0, \underline i), T_{(w_0,\underline i)})\neq 0.$$

Proof of ($\Longleftarrow$): Assume that $a_{1}\neq n-1$ and $a_{2}\leq n-2$. 
By Lemma \ref{7.3} (2),  the natural homomorphism $$H^1(Z(w_0, \underline{i}),T_{(w_0, \underline i)})\longrightarrow H^1(Z(u_1, \underline{i}_1),T_{(u_1, \underline i_1)})$$
of $B$-modules is an isomorphism.
By LES, we have the following exact sequence of $B$-modules:

 $\cdots \longrightarrow H^0(Z(u_1, \underline{i}_1),T_{(u_1, \underline i_1)})\longrightarrow 
 H^0(Z(u'_1,\underline{i'}_1), T_{(u'_1, \underline{i'}_1)})\overset {h_1}\longrightarrow H^1(u_1, \alpha_n)\longrightarrow 
 H^1(Z(u_1, \underline{i}_1),T_{(u_1, \underline i_1)})\longrightarrow H^1(Z(u'_1,\underline{i'}_1), T_{(u'_1, \underline{i'}_1)})
 \longrightarrow 0.$
 
 By Corollary \ref{B}, we see that the natural homomorphism $h_1: H^0(Z(u'_1,\underline{i'}_1), T_{(u'_1, \underline{i'}_1)})\longrightarrow H^1(u_1, \alpha_n)$ is
 surjecive.
 Therefore, the natural homomorphism
 $$H^1(Z(u_1, \underline{i}_1),T_{(u_1, \underline i_1)})\longrightarrow H^1(Z(u'_1,\underline{i'}_1), T_{(u'_1, \underline{i'}_1)})$$ is an isomorphism.
Hence the natural homomorphism $$H^1(Z(w_0, \underline{i}),T_{(w_0, \underline i)})\longrightarrow  H^1(Z(u'_1,\underline{i'}_1), T_{(u'_1, \underline{i'}_1)})$$
is an isomorphism.
By Lemma \ref{lemma7.2} (2),  the natural homomorphism 
$$H^1(Z(u'_1,\underline{i'}_1), T_{(u'_1, \underline{i'}_1)})\longrightarrow H^1(Z(w_k,\underline{j}_k), T_{(w_k, \underline{j}_k)})$$
is an isomorphism. Therefore, the natural homomorphism
$$H^1(Z(w_0, \underline{i}),T_{(w_0, \underline i)})\longrightarrow H^1(Z(w_k,\underline{j}_k), T_{(w_k, \underline{j}_k)})$$ is an isomorphism.
By LES, we have the following exact sequence of $B$-modules:

$\cdots \longrightarrow H^0(Z(w_k,\underline{j}_k), T_{(w_k, \underline{j}_k)})\longrightarrow
H^0(Z(\tau_{k}, \underline j_k'), T_{(\tau_{k}, \underline j_k')})\overset {h_2}\longrightarrow H^1(w_k, \alpha_n)\longrightarrow  H^1(Z(w_k,\underline{j}_k), T_{(w_k, \underline{j}_k)})\overset {h_3}\longrightarrow
H^1(Z(\tau_{k}, \underline j_k'), T_{(\tau_{k}, \underline j_k')})\longrightarrow 0.$

 By Lemma \ref{7.5}(2), we see that the map $h_2: H^0(Z(\tau_{k}, \underline j_k'), T_{(\tau_{k}, \underline j_k')})\longrightarrow H^1(w_k, \alpha_n)$
 is surjecive. Therefore, the map $h_3: H^1(Z(w_k,\underline{j}_k), T_{(w_k, \underline{j}_k)})\longrightarrow
H^1(Z(\tau_{k}, \underline j_k'), T_{(\tau_{k}, \underline j_k')})$ is an isomorphism.

Thus, the natural map $$H^1(Z(w_0, \underline{i}),T_{(w_0, \underline i)})\longrightarrow H^1(Z(\tau_{k}, \underline j_k'), T_{(\tau_{k}, \underline j_k')})$$
is an isomorphism.
By using Lemma \ref{lemma7.2}(2) and Lemma \ref{7.5}(2) repeatedly, we see that the natural map  
$$H^1(Z(\tau_{k}, \underline j_k'), T_{(\tau_{k}, \underline j_k')})\longrightarrow H^1(Z(\tau_{r}, \underline j_r'), T_{(\tau_{r}, \underline j_r')})$$ is an isomorphism
for all  $1\leq r \leq k-1$.

Note that $\tau_1\in W_{S\setminus \{\alpha_n\}}$. 
By taking $u=id$ and $v=\tau_1$ in Lemma \ref{lemma7.2}(2), we see that  $H^1(Z(\tau_{1}, \underline j_1'), T_{(\tau_{1}, \underline j_1')})=0.$
Therefore, we conclude that $H^1(Z(w_0, \underline{i}),T_{(w_0, \underline i)})=0.$ This completes the proof of the theorem.
 \end{proof}

 \begin{corollary} Let $c$ be a Coxeter element such that $c$ is 
of the form $[a_1, n][a_2, a_{1}-1]\cdots [a_k, a_{k-1}-1]$ with 
$a_{1}\neq n-1$, $a_{2}\leq n-2$ and $a_k=1$. Let $(w_0, \underline i)$ be a reduced expression of $w_0$ in terms of $c$ as in Theorem \ref{Theorem}.
Then, 
 $Z(w_0, \underline i)$ has no deformations. %That is, $Z(w_0, \underline i)$ is rigid.
\end{corollary}
\begin{proof}
 By Theorem \ref{Theorem} and by \cite[Proposition 3.1]{CKP}, we have  
 $H^i(Z(w_0, \underline i), T_{(w_0,\underline i)})=0$ for all $i>0$.
Hence, by \cite [p. 272, Proposition 6.2.10]{Huy},
we see that  $Z(w_0, \underline i)$ has no deformations. 
\end{proof}

\end{document}